\RequirePackage[l2tabu, orthodox]{nag}
\documentclass[a4paper,12pt]{amsart}
\usepackage[all, warning]{onlyamsmath}

\usepackage[
  nomarginpar,
  margin=1in,
]{geometry}
\usepackage{mainstyle}

\begin{document}

\title{Global \texorpdfstring{$\bmplus$}{+}-regularity of regular del Pezzo surfaces in mixed characteristic}
\author{Hirotaka Onuki}
% \date{}

\begin{abstract}
  Let $R = W(k)$ be the ring of Witt vectors over an algebraically closed field $k$ of characteristic $p > 2$.
  Let $M$ be a three-dimensional regular integral flat projective $R$-scheme such that $H^0(M,\mathcal{O}_M) = R$ and the anticanonical sheaf $\omega_M^{-1}$ is ample.
  We show that $M$ is globally $+$-regular if the closed fiber $M_k$ is reduced.
\end{abstract}

\maketitle

\tableofcontents

\section{Introduction}
Throughout this introduction, by \emph{surfaces} (resp.\ \emph{curves}) \emph{in mixed characteristic} we mean schemes of \emph{relative dimension} two (resp.\ one) over a complete DVR of mixed characteristic with algebraically closed residue field.

Birational geometry in mixed characteristic has seen substantial progress in recent years.
For example, the minimal model program for surfaces in mixed characteristic $(0, p > 5)$ has been established by~\cite{B+MMP,TakamatsuYoshikawa23}.

One of the main difficulties in this subject is the failure of Kawamata--Viehweg-type vanishing theorems in mixed characteristic~\cite{Totaro24:Terminal3foldsCM}, as well as in positive characteristic~\cite{Raynaud78:Contre}.
In positive characteristic, a key to dealing with this difficulty is the notion of globally $F$-regular varieties~\cite{Smith00:gFreg}, which form an important subclass of globally $F$-split varieties. Recall that a variety in positive characteristic is globally $F$-split if the absolute Frobenius morphism $\sO_X \to F\push \sO_X$ splits globally. It is known from~\cite[Theorem 1.1]{SchwedeSmith10} that globally $F$-regular varieties are of Fano type.
The theory of global $F$-regularity plays a crucial role in the minimal model program in positive characteristic (see, for example,~\cite{HaconXu15}).

Takamatsu and Yoshikawa~\cite{TakamatsuYoshikawa23} and Bhatt et al.~\cite{B+MMP} introduced the notion of \emph{global $\bmplus$-regularity} as a mixed characteristic analog of global $F$-regularity (see \myref{def:gpr} for its definition).
Bhatt's vanishing theorem~\cite{Bhatt21v2:CM} implies that Kawamata--Viehweg-type vanishing holds for globally $\bmplus$-regular schemes~\cite[Corollary 6.12]{B+MMP}.
Globally $\bmplus$-regular schemes are conjectured to be of Fano type, in analogy with the case of globally $F$-regular varieties~\cite[Conjecture 6.17]{B+MMP}.

While some examples of globally $\bmplus$-regular schemes are already implicit in the literature (see, for example, \cite[Corollary 7.5]{B+MMP} and \cite[Theorem 6.4]{Yoshikawa25}), they have not been systematically investigated.

In this thesis, we study the global $\bmplus$-regularity of low-dimensional Fano schemes in mixed characteristic.
Regular Fano curves in mixed characteristic are either $\PP^1$ or the blowup of $\PP^1$ at a closed point, and they are globally $\bmplus$-regular (see \myref{prop:classification-of-2-dim-Fano}).
Smooth del Pezzo surfaces in mixed characteristic are either $\PP^1 \times \PP^1$ or blowups of $\PP^2$ at $r \le 8$ points (see \myref{prop:smooth-wdP}).
Very recently, Yoshikawa essentially proved that they are globally $\bmplus$-regular if the base ring is the ring of Witt vectors~\cite[Theorem 6.4]{Yoshikawa25}.
We therefore turn to the study of regular del Pezzo surfaces whose closed fiber may be singular.
Our main result is the following:
\begin{maintheorem}\label{main:globally-plus-regular}
  Let $R = \W(k)$ be the ring of Witt vectors of an algebraically closed field $k$ of characteristic $p > 0$.
  Let $M$ be a regular integral flat projective $R$-scheme with $\Hh^0(M, \sO_M) = R$.
  Suppose that $\dim M = 3$ and the anticanonical sheaf $\omega_M^\inv$ is ample.
  If $p > 2$ and the closed fiber $M_k$ is reduced, then $M$ is globally $\bmplus$-regular.
\end{maintheorem}
As a consequence of \myref{main:globally-plus-regular}, we deduce a Kawamata--Viehweg-type vanishing theorem for these surfaces from~\cite[Corollary 6.12]{B+MMP}.
\begin{maintheorem}\label{main:vanishing}
  Let $M$ be as in \myref{main:globally-plus-regular}.
  Let $\sA$ be a big and semiample line bundle on $M$.
  Then $\Hh^i(M, \omega_M \otimes \sA) = 0$ holds for all $i > 0$.
\end{maintheorem}
It is reasonable to expect that the reducedness of $M_k$ in \myref{main:globally-plus-regular} can be derived from the other assumptions.
One may also ask whether the assumption $p > 2$ is essential.

\subsection{Idea of the proof of \myref{main:globally-plus-regular}}
Our proof of \myref{main:globally-plus-regular} is based on the theory of \emph{quasi-$F$-splitting} in mixed characteristic, together with the classification of the closed fiber $M_k$.

Yoshikawa~\cite{Yoshikawa25} introduced the notion of quasi-$F$-splitting in mixed characteristic, based on Yobuko's theory in positive characteristic~\cite{Yobuko19} (see \myref{def:mix-quasi-F-split} for its definition).
He essentially showed that quasi-$F$-splitting implies global $\bmplus$-regularity in a certain setting~\cite[Theorem 6.3]{Yoshikawa25}, and we use a generalization of this result (\myref{thm:qFs-implies-gpr}).

We split the proof into cases depending on the singularities of the closed fiber $M_k$.

If the closed fiber $M_k$ has only rational double points (RDP), then $M$ is quasi-$F$-split by~\cite[Theorem~B]{OnukiTakamatsuYoshikawa25v2}.

If the closed fiber $M_k$ is normal but not an RDP del Pezzo surface, then $M_k$ is the projective cone over an elliptic curve by~\cite{HidakaWatanabe81}.
Such a variety is easily seen to be quasi-$F$-split, and hence $M$ is quasi-$F$-split.

If the closed fiber $M_k$ is nonnormal, we give a classification of $M_k$ (\myref{thm:irreducible-main} and \ref{thm:reducible-main}), inspired by~\cite{Fujita90}.
Fujita~\cite{Fujita90} (cf.~\cite{Fukuoka20:dP6}) investigated singular fibers of del Pezzo fibrations over curves in characteristic zero, and our classification result is analogous to his.
His classification is based on Mori's classification of three-dimensional contractions~\cite{Mori82}, which is not available in mixed characteristic, and we instead use the classification of nonnormal Gorenstein del Pezzo surfaces by Reid~\cite{Reid94}.
Using this classification of $M_k$, together with the criterion for $F$-splitting by~\cite{MillerSchwede12}, we conclude that $M_k$ is globally $F$-split. In particular, $M_k$ is quasi-$F$-split.

\subsection{Outline of the paper}
\myref{sec:Preliminaries} presents some preliminaries.
\myref{sec:qFs-gpr} provides a sufficient condition for global $\bmplus$-regularity.
In \myref{sec:curves-and-smooths}, we discuss classification results for some low-dimensional Fano schemes in mixed characteristic.
In \myref{sec:qFs-dP-over-fields}, we study the quasi-$F$-splitting of Gorenstein del Pezzo surfaces in equal characteristic.
In \myref{sec:main}, we classify the nonnormal closed fibers of regular del Pezzo surfaces in mixed characteristic, showing \myref{main:globally-plus-regular}.
\myref{sec:qFs-gpr} is based on discussions with Shou Yoshikawa.

\subsection*{Acknowledgements}
I would like to express my deepest gratitude to my supervisor, Shunsuke Takagi, for his constant guidance, insightful advice, and patient support.
I am greatly indebted to Shou Yoshikawa for many helpful discussions, for generously sharing his ideas, and for allowing me to include \myref{thm:qFs-implies-gpr}.
I would like to express my gratitude to Ryotaro Iwane, Teppei Takamatsu and Hiromu Tanaka for many valuable discussions.

\section{Preliminaries}\label{sec:Preliminaries}
First, we summarize notation used in this paper.
\begin{itemize}
  \item All schemes are assumed to be separated.
  \item For a scheme $M$ and a prime number $p$, write $\modp{M}$ for the closed subscheme of $M$ defined by $p \in \sO_M$.
  \item Denote by
  \[
    \sectionRing{M}{\sA} \coloneqq \bigoplus_{i \ge 0} \Hh^0(M, \sA\powotimes{i})
  \]
  the section ring of an ample line bundle $\sA$ on a scheme $M$.
  \item Let $(R, \frakm)$ be a Noetherian local ring and let $M$ be a projective $R$-scheme. Set
  \[
    \Hh^i_\frakm(M, \sF) \coloneqq \Hh^i R\Gamma_\frakm R\Gamma(M, \sF)
  \]
  for $i \ge 0$ and a coherent sheaf $\sF$ on $M$.
  \item Let $(R, \frakm)$ be an excellent local ring with a normalized dualizing complex $\omega_R^\bullet$. For a (separated) morphism $f \colon M \to \Spec R$ of finite type, we set $\omega_M^\bullet \coloneqq f^!\omega_R^\bullet$. We refer to~\citestacks{08XG} for the details (cf.~\cite[Subsection~2.1]{B+MMP}).
  \item A Noetherian scheme $M$ is said to be \emph{\SerreS{2}} if it satisfies Serre's second condition: that is, every point of codimension at least $2$ or $1$ has depth at least $2$ or $1$, respectively.
  \item Let $k$ be a field. For a nonnegative integer $a$, let $\rationalruled{a}$ denote the rational ruled surface $\PP_{\PP^1_k}(\sO_{\PP^1_k} \oplus \sO_{\PP^1_k}(-a))$.
  \item Let $X$ be a scheme of positive characteristic. Write $F \colon X \to X$ for the absolute Frobenius morphism of $X$.
\end{itemize}

\subsection{Global \texorpdfstring{$\bmplus$}{+}-regularity}
\begin{definition}[\cite{B+MMP,TakamatsuYoshikawa23}]\label{def:gpr}
  Let $M$ be an excellent normal integral scheme with a dualizing complex. Assume that every closed point of $M$ has positive residue characteristic.
  Then $M$ is said to be \emph{globally $\bmplus$-regular} if for every finite surjection $g \colon N \to M$ from a normal integral scheme $N$, the morphism $\sO_M \to g\push \sO_N$ splits as a morphism of $\sO_M$-modules.
\end{definition}

\begin{lemma}
  Let $M$ be a normal integral projective scheme over a Noetherian local ring $(R, \frakm_R)$.
  Let $\sA$ be an ample line bundle on $M$, and let $S \coloneqq \sectionRing{M}{\sA}$.
  Set $\frakm_S = \frakm_R S + S_{> 0}$.
  If $S_{\frakm_S}$ is (globally) $\bmplus$-regular, then so is $M$.
\end{lemma}

\begin{proof}
  Let $g \colon N \to M$ be a finite surjection from a normal integral scheme $N$.
  Set $S' \coloneqq \sectionRing{N}{g\pull \sA}$.
  Since $S_{\frakm_S}$ is globally $\bmplus$-regular, $S_{\frakm_S} \to (S')_{\frakm_S}$ splits.
  Using~\cite[Corollary~3.6.7]{BrunsHerzog}, we see that $S \to S'$ splits. Hence, $\sO_M = \tilde{S} \to g\push \sO_N = \tilde{S'}$ splits.
\end{proof}

% \begin{definition}[{\cite[Definition 6.9]{MaSchwede21}}]
%   Let $R$ be a Noetherian complete local domain.
%   An $R$-algebra $B$ is said to be a \emph{(balanced) big Cohen--Macaulay (BCM)} $R$-algebra if every system of parameters $\underline{x} = x_1,\dots,x_d$ of $R$ is a regular sequence on $B$.
%   Then $R$ is said to be \emph{perfectoid BCM-regular} if for every perfectoid BCM $R$-algebra $B$, $R \to B$ splits as a morphism of $R$-modules (see \cite[Definition 3.5]{BhattMorrowScholze18:Integralpadic} for the definition of a \emph{perfectoid ring}).
%   A scheme is said to be \emph{perfectoid BCM-regular} if $\completion{\sO_{M, x}}$ is perfectoid BCM-regular for every closed point $x \in M$.
% \end{definition}

% Note that any complete regular local ring is perfectoid BCM-regular by \cite[Theorem E]{MaSchwede21}.

\begin{definition}
  Let $M$ be an excellent integral scheme.
  Let $K$ be the function field of $M$, and let $\overline{K}$ be an algebraic closure of $K$.
  The \emph{absolute integral closure} of $M$ is the integral closure $M^+$ of $M$ in $\overline{K}$.

  Suppose that $M$ is normal.
  Let $D$ be a $\QQ$-Weil divisor on $M$.
  Then we define
  \[
    \sO_{M^+}(D) \coloneqq \colim_{g \colon N \to M} \sO_N(g\pull D)\text,
  \]
  where $g$ runs over all finite surjective morphisms from a normal integral scheme $N$ equipped with an inclusion $K(N) \into K(M^+) = \overline{K}$.
  Note that the pullback $g\pull D$ is defined since $g$ is a finite morphism.
\end{definition}

\subsection{Nonnormal schemes}
Let $X$ be a reduced equidimensional excellent scheme, and let $\nu \colon Y = X\normaln \to X$ be its normalization.
Define $\conductorideal \coloneqq (\colonideal{\sO_X}{\nu\push \sO_Y}) \subset \sO_X$ to be the \emph{conductor ideal} of $X$. It is also an ideal sheaf of the $\sO_X$-algebra $\nu\push \sO_Y$, and it satisfies $\conductorideal = \nu\push \conductorideal\sO_Y$.
Set $C \subset X$ and $\BY \subset Y$ to be the closed subschemes defined by $\conductorideal$ and $\conductorideal\sO_Y$, respectively.
The diagram
\begin{equation*}
  \begin{tikzcd}
    \BY \arrow[dr, phantom, "\lrcorner" very near start, "\ulcorner" very near end]
    \ar[d, "\nu\restrict{\BY}"'] \ar[r, hook] & Y \ar[d, "\nu"] \\
    C \ar[r, hook] & X
  \end{tikzcd}
\end{equation*}
is called the \emph{conductor square} and is both cartesian and cocartesian in the category of schemes. This can be seen from the exact sequence
\begin{equation*}
  0 \to \sO_X \to \nu\push \sO_Y \oplus \sO_C \to \nu\push \sO_\BY \to 0\text,
\end{equation*}
where the left map is the diagonal map, and the right map is the difference map.
If $X$ is \SerreS{2}, then $\BY$ is an effective Weil divisor on $Y$, and $C$ is of pure codimension $1$ in $X$ (see \cite[Proposition 2.2]{Reid94} and \cite[Appendix~A]{FanelliSchroeer20}).

\subsection{Gorenstein del Pezzo surfaces}
We review the classification results for possibly nonnormal Gorenstein del Pezzo surfaces, following Reid~\cite{Reid94} and Hidaka--Watanabe~\cite{HidakaWatanabe81}.
Throughout this subsection, we work over an algebraically closed field $k$ of characteristic $p$.

\begin{definition}
  A connected reduced projective $k$-scheme $X$ is said to be a \emph{Gorenstein del Pezzo surface} if it is of dimension $2$, Gorenstein, and $\omega_X^\inv$ is ample. The intersection number $d \coloneqq (\omega_X^\inv)\powndot{2}$ is called the \emph{degree} of $X$. A Gorenstein del Pezzo surface with only rational double points (RDP) is called an \emph{RDP del Pezzo surface}.
\end{definition}

\begin{theorem}[{\cite[Theorem 2.2]{HidakaWatanabe81}}]\label{thm:normal-dP-classification}
  Let $X$ be a normal Gorenstein del Pezzo surface. Then one of the following holds:
  \begin{enumerate}
    \item $X$ is an RDP del Pezzo surface.
    \item For some elliptic curve $E$ and some ample line bundle $\sA$ on $E$, $X$ is isomorphic to $\Proj \sectionRing{E}{\sA}[t]$, where $t$ is an indeterminate of degree $1$.
  \end{enumerate}
\end{theorem}

\begin{proof}
  Let $\widetilde{X}$ be the minimal resolution of $X$.
  By \cite[Theorem 2.2]{HidakaWatanabe81} and its proof, either (1) holds, or else $\widetilde{X} \isom \PP_E(\sO_E \oplus \sA)$ for some elliptic curve $E$ and some ample line bundle $\sA$ on $E$. In the latter case, we have $\sectionRing{X}{\omega_X^\inv} \isom \sectionRing{E}{\sA}[t]$ as shown in the proof of~\cite[Proposition 4.2]{HidakaWatanabe81}.
\end{proof}

Next, we move on to the nonnormal case.

\begin{notation}\label{nota:nonnormal-dP}
  Suppose that $X$ is a nonnormal Gorenstein del Pezzo surface.
  Let $\nu \colon Y \to X$ be its normalization, and let
  \begin{equation*}
    \begin{tikzcd}
      \BY % \arrow[dr, phantom, "\lrcorner" very near start, "\ulcorner" very near end]
      \ar[d, "\nu\restrict{\BY}"'] \ar[r, hook] & Y \ar[d, "\nu"] \\
      C \ar[r, hook] & X
    \end{tikzcd}
  \end{equation*}
  be the conductor square.
  Let
  \begin{equation*}
    Y = \coprod_{1 \le i \le r} Y_i
  \end{equation*}
  be the decomposition into connected components and let $\BY_i = \BY \cap Y_i \subset Y_i$.
  Set $\sO_X(1) = \omega_X^\inv$, $\sO_Y(1) = \nu\pull \omega_X^\inv$ and $\sO_{Y_i}(1) = (\nu\restrict{Y_i})\pull \omega_X^\inv$. Note that all of these are ample line bundles since $X$ is a Gorenstein del Pezzo surface and $\nu$ is a finite morphism.
  Set $d_i = (\sO_{Y_i}(1))\powndot{2}$ and $d = \sum_i d_i = (\sO_X(1))\powndot{2}$.
\end{notation}
In this setting, since $X$ is the pushout of the diagram $C \leftarrow \BY \to Y$, the classification of nonnormal Gorenstein del Pezzo surfaces is reduced to that of such diagrams.
Reid first classifies each component $(Y_i, \sO_{Y_i}(1), \BY_i)$:
\begin{theorem}[{\cite[Theorem 1.1]{Reid94}}]\label{thm:normalization-dP-classification}
  Let $X$ be a nonnormal Gorenstein del Pezzo surface. We use the notation of \myref{nota:nonnormal-dP}.
  Then each $\BY_i$ is isomorphic to a plane conic (that is, a smooth conic, a line pair or a double line in $\PP^2$) and the possibilities for $(Y_i, \sO_{Y_i}(1), \BY_i)$ are as in~\myref{tab:normalization}.
  \begin{table}[h]
    \centering
    \caption{Reid's table. For cases (d) and (e), let $F$ be a fiber and $E$ the negative section of the rational ruled surface.}
    \label{tab:normalization}
    \begin{tblr}{cccccc}
      \hline
      Case & $Y_i$ & $\sO_{Y_i}(1)$ & $\sO_{Y_i}(\BY_i)$ & $d_i$ & $\BY_i$ \\ \hline\hline
      & & & & & smooth conic \\
      (a) & $\PP^2$ & $\sO_{\PP^2}(1)$ & $\sO_{\PP^2}(2)$ & $1$ & line pair \\
      & & & & & double line \\ \hline
      (b) & $\PP^2$ & $\sO_{\PP^2}(2)$ & $\sO_{\PP^2}(1)$ & $4$ & smooth conic \\ \hline
      & & & & \SetCell[r=2]{m} $\ge 2$ & line pair \\
      (c) & $\PP(1,1,d)$ & $\sO_{\PP(1,1,d)}(d)$ & $\sO_{\PP(1,1,d)}(2)$ & & double line \\ \cline{5-6}
      & & & & $2$ & smooth conic \\ \hline
      \SetCell[r=2]{m} (d) & \SetCell[r=2]{m} $\rationalruled{d-2}$ & \SetCell[r=2]{m} $\sO_{\rationalruled{d-2}}((d - 1)F + E)$ & \SetCell[r=2]{m} $\sO_{\rationalruled{d-2}}(F + E)$ & $\ge 2$ & line pair \\ \cline{5-6}
      & & & & $2$ or $3$ & smooth conic \\ \hline
      (e) & $\rationalruled{d-4}$ & $\sO_{\rationalruled{d-4}}((d - 2)F + E)$ & $\sO_{\rationalruled{d-4}}(E)$ & $\ge 4$ & smooth conic \\ \hline
    \end{tblr}
  \end{table}
\end{theorem}

\begin{remark}\label{rmk:Reid-table}
  We give a few remarks regarding \myref{tab:normalization}.
  \begin{itemize}
    \item We have $\sO_{Y_i}(1) \ndot \BY_i = 2$.
    \item $\sO_{Y_i}(1)$ is very ample and defines an embedding $Y_i \into \PP^{d_i + 1}$.
    \item The choice of a conic $\BY_i \subset Y_i$ of a fixed type is unique up to isomorphism. For example, any two smooth conics in $\PP^2$ are mapped to each other by an automorphism of $\PP^2$, and a similar remark applies to the other cases of $(Y_i, \BY_i)$.
  \end{itemize}
\end{remark}

Reid classified \emph{tame} nonnormal Gorenstein del Pezzo surfaces in the following sense.
\begin{definition}[{\cite[4.7]{Reid94}}]\label{def:tame}
  Let $X$ be a Gorenstein del Pezzo surface and we use the notation of \myref{nota:nonnormal-dP}.
  Then $X$ is said to be \emph{tame} if either $\BY$ is reduced or $C\redn$ is smooth.
  By convention, we say that normal Gorenstein del Pezzo surfaces are tame.
\end{definition}

% We remark that Reid defined tameness only for nonnormal Gorenstein del Pezzo surfaces but the author believes that the extension to normal Gorenstein del Pezzo surfaces is natural.

\begin{theorem}[{\cite[Corollary 4.10 (1)]{Reid94}}]
  Let $X$ be a Gorenstein del Pezzo surface.
  Then $X$ is tame if and only if $\Hh^1(X, \sO_X) = 0$.
\end{theorem}

\begin{proof}
  When $X$ is normal, then $\Hh^1(X, \sO_X) = 0$ holds by \cite[Corollary~2.5]{HidakaWatanabe81} and $X$ is tame by definition. When $X$ is not normal, the assertion is proved in~\cite[Corollary 4.10 (1)]{Reid94}.
\end{proof}

We divide the classification into two cases: the irreducible case and the reducible case.

\begin{theorem}[{\cite[Theorem 1.5]{Reid94}}]\label{thm:irreducible-dP-classification}
  Let $X$ be an irreducible, nonnormal, tame Gorenstein del Pezzo surface, and we use the notation of \myref{nota:nonnormal-dP}.
  Recall that $\BY$ is isomorphic to a plane conic.
  Then the following hold.
  \begin{enumerate}
    \item $C \isom \PP^1$ and $\nu\restrict{\BY} \colon \BY \to C$ is a linear projection of the plane conic $\BY$ onto the line $C$. In particular, $\nu\restrict{\BY}$ is a double cover.
    \item The isomorphism class of $X$ is determined by $(Y, \BY)$, unless $p = 2$. If $p = 2$ and $\BY$ is a smooth conic, there are two isomorphism classes of $X$, depending on whether $\nu\restrict{\BY}$ is inseparable.
    % \item $X$ is one of the surfaces given in \myref{exa:Reid-irreducible} below.
  \end{enumerate}
\end{theorem}

% \begin{proof}
%   Reid proves the latter, which is (irreducible part of) \cite[Theorem 1.5]{Reid94}, as a consequence of the former, which essentially appears in the proof.
% \end{proof}

\begin{example}[{\cite[1.3 and 1.4]{Reid94}}]\label{exa:Reid-irreducible}
  We present examples of irreducible, nonnormal, tame Gorenstein del Pezzo surfaces via anticanonical embeddings.
  \begin{enumerate}
    \item Let $(Y, \BY) = (Y_1, \BY_1)$ be a pair from \myref{tab:normalization} such that $d \ge 3$. Embed $Y$ into $\PP^{d + 1}$ by $\sO_{Y}(1)$ and let $\Pi \isom \PP^2$ be the plane containing $\BY$. Take a linear projection $Y \to X \subset \PP^d$ from a point $y \in \Pi \setminus \BY$. Then $X$ is a Gorenstein del Pezzo surface of degree $d \ge 3$.
    \item We give more explicit descriptions when $(Y, \BY)$ is of case (c). If $\BY$ is a line pair, $X$ is the projective cone of a nodal rational curve of degree $d \ge 3$: for example, $X = \Proj k[s^d+t^d, s^{d-1}t, s^{d-2}t^2, \dots, st^{d-1}, u] \subset \PP^d$. If $\BY$ is a double line, $X$ is the projective cone of a cuspidal rational curve of degree $d \ge 3$: for example, $X = \Proj k[s^d, s^{d-1}t, s^{d-2}t^2, \dots, s^2t^{d-2}, t^d, u] \subset \PP^d$. Note that the embedding dimension of each of these cones at the vertex is $d$.
    \item Nonnormal tame Gorenstein del Pezzo surfaces of degree $d = 1$ (resp.\ $2$) can be described as weighted hypersurfaces of degree $6$ (resp.\ $4$) in $\PP(1,1,2,3)$ (resp.\ $\PP(1,1,1,2)$). See \cite[1.4]{Reid94}.
  \end{enumerate}
\end{example}

\begin{theorem}[{\cite[Theorem 1.5]{Reid94}}]\label{thm:reducible-dP-classification}
  Let $X$ be a reducible tame Gorenstein del Pezzo surface, and we use the notation of \myref{nota:nonnormal-dP}.
  Then the following hold.
  \begin{enumerate}
    \item All $\BY_i$ are isomorphic.
    \item $X$ is one of the surfaces given in \myref{exa:Reid-reducible} below.
  \end{enumerate}
\end{theorem}

% A statement in the form of a description for $C \leftarrow \BY \to Y$ (see for the case of double lines) is also possible but I believe that the following table is as simple.

\begin{example}[{\cite[1.3 and 1.4]{Reid94}}]\label{exa:Reid-reducible}
  We present examples of reducible tame Gorenstein del Pezzo surfaces via anticanonical embeddings. We follow the notation of \myref{nota:nonnormal-dP}.
  \begin{enumerate}
    \item Let $(Y_1, \BY_1)$ and $(Y_2, \BY_2)$ be pairs from \myref{tab:normalization} with $\BY_1$ and $\BY_2$ smooth conics, and set $d = d_1 + d_2$. Assume that not both are from case (a) (so that $d \ge 3$). Each $Y_i$ is embedded into $\PP^{d_i + 1}$ by $\sO_{Y_i}(1)$. Embed each $\PP^{d_i + 1}$ into $\PP^d$ so that they intersect only in a plane $\Pi \isom \PP^2 \subset \PP^d$ and both $\BY_i$ map to a common smooth conic $C$ in $\Pi$. Then $X = X_1 \cup X_2 \subset \PP^d$ is a tame Gorenstein del Pezzo surface, where $X_i$ is the image of $Y_i$.
    \item Let $(Y_i, \BY_i)$ ($1 \le i \le r$, $r \ge 2$) be pairs from \myref{tab:normalization} with all $\BY_i$ line pairs. Write $\BY_i = \Gamma_i \cup \Gamma'_i$. Assume that when $r = 2$, not both are from case (a) (so that $d \ge 3$). Fix $\PP^r \subset \PP^d$, choose homogeneous coordinates $[x_0:\dots:x_r]$ on $\PP^r$, and set $O \coloneqq [1:0:\dots:0] \in \PP^r$. Let $C \subset \PP^r$ be the union of the coordinate lines $L_i \coloneqq (x_j = 0 \text{ for } j \neq 0,i)$ through $O$. We regard the indices $i$ cyclically: we set $L_0 \coloneqq L_r$. Embed each $\PP^{d_i + 1}$ into $\PP^d$ so that $\Gamma_i$ and $\Gamma'_i$ are sent to $L_{i-1}$ and $L_i$, respectively, and such that $\PP^{d_i + 1}$ intersects the linear span of $\bigcup_{l \neq i}\PP^{d_l + 1}$ only in the plane spanned by $L_{i-1}$ and $L_i$. Then $X = \bigcup X_i \subset \PP^d$ is a tame Gorenstein del Pezzo surface, where $X_i$ is the image of $Y_i$.
    \item Let $(Y_i, \BY_i)$ ($1 \le i \le r$, $r \ge 2$) be pairs from \myref{tab:normalization} with all $\BY_i$ double lines. Assume that when $r = 2$, not both are from case (a) (so that $d \ge 3$). Let $\Gamma$ be a line in $\PP^r \subset \PP^d$, and let $C$ be the first-order neighborhood of $\Gamma$ in $\PP^r$, defined by $\sI_\Gamma^2$. Let $\Pi_1, \dots, \Pi_r$ be planes in $\PP^r$ containing $\Gamma$ such that any $r-1$ of these span $\PP^r$. Embed $\PP^{d_i + 1}$ into $\PP^d$ so that all $(\BY_i)\redn$ map to $\Gamma$ and so that each $\PP^{d_i + 1}$ intersects the linear span of $\bigcup_{l \neq i}\PP^{d_l + 1}$ only in $\Pi_i$. Then $X = \bigcup X_i \subset \PP^d$ is a tame Gorenstein del Pezzo surface, where $X_i$ is the image of $Y_i$.
    \item Let $Y_1 = Y_2 = \PP^2$ and let $\BY_1 = \BY_2$ be a plane conic $(q = 0)$. Then $X = (w(w - q) = 0) \subset \PP(1,1,1,2)$ is a tame Gorenstein del Pezzo surface of degree $2$.
  \end{enumerate}
\end{example}

\begin{remark}\label{rmk:remark-on-reducible}
  Under the setting of \myref{thm:reducible-dP-classification}, it follows that $Y_i \to X_i$ and $\BY_i \to C \cap X_i$ are isomorphisms.
\end{remark}

\subsection{Ring of Witt vectors}
Throughout this subsection, we fix a prime number $p$.
We recall the definition of the ring of ($p$-typical) Witt vectors (see \cite[II,~\S~6]{Serre79:Local-fields} and \cite[Section~1.3]{Illusie79:deRhamWitt} for the details).
Note that we consider possibly non-Noetherian rings that are not necessarily of characteristic $p$.

For $n \ge 0$, write
\[
  \varphi_n(X_0,\dots,X_n) \coloneqq X_0^{p^n} + pX_1^{p^{n-1}} + \dots + p^nX_n \in \ZZ[X_0,\dots,X_n] \text{.}
\]

Let $A$ be a ring not necessarily of characteristic $p$.
Then we write $\W(A)$ for the \emph{ring of ($p$-typical) Witt vectors} of $A$. We state its basic properties. We have
\[
  \W(A) = \prod_{i = 0}^\infty A
\]
as a set, and the map
\[
  \varphi \colon \W(A) \to \prod_{n = 0}^\infty A, \quad (a_0,a_1,a_2,\dotsc) \mapsto \bigl(\varphi_0(a_0), \varphi_1(a_0,a_1), \varphi_2(a_0,a_1,a_2), \dotsc\bigr)
\]
is a ring homomorphism.
For every $a \in A$, we write $[a] = (a,0,0,\dotsc) \in \W(A)$.
We have natural maps
\[
  \text{$V \colon \W(A) \to \W(A)$, $R \colon \W(A) \to \W(A)$, and $F \colon \W(A) \to \W(A)$}
\]
that satisfy
\begin{gather*}
  V(a_0,a_1,\dotsc) = (0,a_0,a_1,\dotsc), \quad R(a_0,a_1,a_2,\dotsc) = (a_1,a_2,\dotsc)\text, \\
  \varphi(F(a_0,a_1,a_2,\dotsc)) = \bigl(\varphi_1(a_0,a_1),\varphi_2(a_0,a_1,a_2),\dotsc\bigr)\text{.}
\end{gather*}
Then, $F \colon \W(A) \to \W(A)$ is a ring homomorphism. $V \colon F_*\W(A) \to \W(A)$ and $R \colon \W(A) \to \W(A)$ are $\W(A)$-module homomorphisms.

For $n \ge 1$, we define a ring $\W_n(A)$ by
\[
  \W_n(A) \coloneqq \W(A) / V^n\W(A)\text,
\]
which is called the \emph{ring of ($p$-typical) Witt vectors of length $n$}.
We have $\W_n(A) \isom \prod_{i = 0}^{n - 1} A$ as sets.
Then, for every $n \ge 1$, we obtain a ring homomorphism $F \colon \W_{n + 1}(A) \to \W_n(A)$, and we have $\W_{n+1}(A)$-module homomorphisms
\[
  \text{$V \colon F\push \W_{n}(A) \to \W_{n+1}(A)$ and $R \colon \W_{n+1}(A) \to \W_{n}(A)$.}
\]

We define a sheaf of rings $\W_n\sO_M$ by setting
\[
  \W_n\sO_M(U) \coloneqq W_n(\sO_M(U))
\]
for every open set $U \subset M$.
Then $\W_nM \coloneqq (\abs{M},\W_n\sO_M)$ is a scheme.

\subsection{Quasi-\texorpdfstring{$F$}{F}-splitting}\label{subsec:qfs}
We briefly recall the theory of quasi-$F$-splitting in positive and mixed characteristic~\cite{Yobuko19,Yoshikawa25}.
We still fix a prime number $p$.

First, we review the definition of global $F$-splitting.
\begin{definition}[\cite{MehtaRamanathan85,Smith00:gFreg,SchwedeSmith10}]
  Let $X$ be a reduced $F$-finite Noetherian scheme of characteristic $p > 0$.
  \begin{itemize}
    \item $X$ is said to be \emph{(globally) $F$-split} if the map $\sO_X \to F\push \sO_X$ splits as an $\sO_X$-module homomorphism.
  \end{itemize}
  Suppose that $X$ is normal.
  \begin{itemize}
    \item Let $B$ be an effective $\ZZ_{(p)}$-Weil divisor on $X$ such that $\Kk_X + B$ is $\ZZ_{(p)}$-Cartier. The pair $(X, B)$ is said to be \emph{(globally) $F$-split} if, for some $e > 0$, $(p^e - 1)B$ is Cartier and the natural map $\sO_X \to F^e\push \sO_X \to F^e\push \sO_X((p^e - 1)B)$ splits.
    \item $X$ is said to be \emph{globally $F$-regular} if, for every effective Weil divisor $D$ on $X$, there exists some $e > 0$ such that the natural map $\sO_X \to F^e\push \sO_X \to F^e\push \sO_X(D)$ splits.
  \end{itemize}
\end{definition}

% \begin{remark}
%   Let $X$ be a normal $F$-finite Noetherian scheme of characteristic $p > 0$.
%   If $X$ is globally $F$-regular, then $X$ is globally $\bmplus$-regular.
%   The converse is also expected to hold~\cite[Remark~6.16]{B+MMP}.
%   This holds if $X$ is $\QQ$-Gorenstein and $X = \Spec A$ for a local ring $A$~\cite{Singh99}.
%   We also note that if $X$ is globally $\bmplus$-regular, then $X$ is globally $F$-split.
% \end{remark}

Next, we recall the definition of quasi-$F$-splitting in positive characteristic, generalizing the notion of global $F$-splitting.
\begin{definition}[{\cite{Yobuko19}}]
  Let $X$ be a reduced $F$-finite Noetherian scheme of characteristic $p > 0$.
  We define a module $Q_{X,n}$ and a map $\Phi_{X, n}$ by the following pushout diagram of $W_n\sO_X$-modules: 
  \begin{equation*}
    \begin{tikzcd}
      W_n\sO_X \arrow{r}{F} \arrow[two heads]{d}{R^{n-1}} & F_* W_n \sO_X  \arrow[two heads]{d}\\
      \sO_X \arrow{r}{\Phi_{X, n}}& \arrow[lu, phantom, "\ulcorner" , very near start, xshift=0.6em] Q_{X, n}
    \end{tikzcd}\text{.}
  \end{equation*}
  Here, the natural map $F \colon W_n\sO_X \to F_* W_n \sO_X$ is induced since $X$ is of characteristic $p$.
  Note that $Q_{X,n}$ is an $\sO_X$-module by~\cite[Proposition 2.9]{KTTWYY:qFs1}.
  \begin{itemize}
    \item $X$ is said to be \emph{$n$-quasi-$F$-split} if the map $\Phi_{X, n}$ splits as an $\sO_X$-module homomorphism.
    \item $X$ is said to be \emph{quasi-$F$-split} if it is $n$-quasi-$F$-split for some $n \ge 1$.
  \end{itemize}
\end{definition}

Note that $X$ is $1$-quasi-$F$-split if and only if $X$ is globally $F$-split.

\begin{lemma}[{cf.~\cite[Theorem~4.7]{Schwede09:DB}}]\label{lem:qFs-implies-weakly-normal}
  Let $X$ be a reduced $F$-finite Noetherian scheme of characteristic $p$.
  If $X$ is quasi-$F$-split, then it is weakly normal.
\end{lemma}

See \cite{GrecoTraverso80,ReidRobertsSingh96:weaksubintegrality,AndreottiBombieri69} for the definition of seminormality and weak normality (see also~\cite[Section~3]{Schwede09:DB}).

\begin{proof}
  We may assume that $X$ is an affine scheme.
  Write $X = \Spec A$ and let $K$ be the total quotient ring of $A$.
  Let $b \in K$, and assume $b^p \in A$.
  By \cite[Theorem 4.3 and 6.8]{ReidRobertsSingh96:weaksubintegrality}, it suffices to show that $b \in A$.
  Since $A$ is quasi-$F$-split, a splitting $\alpha$ of $\Phi_{A, n}$ exists.
  We have the pushout diagram
  \[
    \begin{tikzcd}[row sep = small]
      W_nA \arrow{r}{F} \arrow[d, "R^{n-1}"'] & F_* W_n A  \arrow{d} \ar[r, phantom, "\ni"] &[-.5cm] F\push [b^p] \ar[d, maps to] \\
      A \arrow{r}{\Phi_{A, n}}
      & \arrow[lu, phantom, "\ulcorner" , very near start, yshift=0em, xshift=0.6em] Q_{A, n} \ar[r, phantom, "\ni"] \ar[l, bend left, "\alpha"]
      & \overline{F\push [b^p]}
    \end{tikzcd}\text{.}
  \]
  Hence, by tensoring it with $W_nK$, we obtain the diagram
  \[
    \begin{tikzcd}[row sep = small]
      &[-.5cm] W_nK \arrow{r}{F} \arrow[d, "R^{n-1}"'] & F_* W_n K  \arrow{d}
      &[+.5cm] [b] \ar[d, maps to] \ar[r, maps to] & F\push [b^p] \ar[d, maps to] \\
      & K \arrow{r}{\Phi_{K, n}}
      & \arrow[lu, phantom, "\ulcorner" , very near start, yshift=0em, xshift=0.6em] Q_{K, n} \ar[l, bend left, "\alpha \otimes 1"]
      & b \ar[r, maps to, dashed] & \overline{F\push [b^p]}
    \end{tikzcd}
  \]
  by~\cite[(1.5.3)]{Illusie79:deRhamWitt}.
  It follows that $(\alpha \otimes 1)(\overline{F\push [b^p]}) = b$.
  Therefore, we obtain $b = \alpha(\overline{F\push [b^p]}) \in A$, as desired.
\end{proof}

We move on to the mixed characteristic case.
We restrict ourselves to the normal case.
\begin{definition}
Let $M$ be a normal scheme that is flat over $\Spec \ZZ_{(p)}$.
We define a module $Q_{M,n}$ and a map $\wt{\Phi_{M,n}}$ by the following pushout diagram of $W_{n+1}\sO_M$-modules:
\[
  \begin{tikzcd}
    \W_{n+1}\sO_M \arrow{r}{F} \arrow[two heads]{d}{R^{n}} & F\push \W_n \sO_M  \arrow[two heads]{d}\\
    \sO_M \arrow{r}{\wt{\Phi_{M,n}}}& \arrow[lu, phantom, "\ulcorner" , very near start, xshift=0.6em] Q_{M, n}
  \end{tikzcd}\text{.}
\]
Note that $\wt{\Phi_{M, n}} \colon \sO_M \to Q_{M, n}$ is an $\sO_M$-module homomorphism.
\end{definition}

\begin{proposition}[{\cite[Proposition~3.2]{OnukiTakamatsuYoshikawa25v2}, cf.~\cite[Proposition 2.9]{KTTWYY:qFs1}}]
Let $M$ be a normal scheme that is flat over $\Spec \ZZ_{(p)}$.
Then we have the following two exact sequences for every $m, n \in \ZZ$:
\begin{gather}
  0 \to F_*W_n\sO_M \xrightarrow{\cdot p} F_*W_n\sO_M \to Q_{M,n} \to 0, \label{eq:exact-p}\\
  0 \to F^n_*Q_{M,m} \xrightarrow{V^n} Q_{M,n+m} \xrightarrow{R^m} Q_{M,n} \to 0. \label{eq:exact-V-R}
\end{gather}
\end{proposition}

\begin{definition}
Let $M$ be a normal scheme that is flat over $\Spec \ZZ_{(p)}$.
By \eqref{eq:exact-p}, the map $\wt{\Phi_{M,n}}$ induces an $\sO_{\modp{M}}$-module homomorphism $\Phi_{M,n} \colon \sO_{\modp{M}} \to Q_{M,n}$.
\end{definition}

\begin{definition}[]
\label{def:mix-quasi-F-split}
Let $M$ be a normal Noetherian scheme flat over $\Spec \ZZ_{(p)}$ such that $\modp{M}$ is $F$-finite.
\begin{itemize}
    \item $M$ is said to be \emph{$n$-quasi-$F$-split} if the map $\Phi_{M, n}$ splits as an $\sO_{\modp{M}}$-module homomorphism.
    \item $M$ is said to be \emph{quasi-$F$-split} if it is $n$-quasi-$F$-split for some $n \ge 1$.
  \end{itemize}
\end{definition}

\begin{lemma}[{\cite[Proposition~4.6]{Yoshikawa25}}]\label{lem:mod-p-qFs}
  Let $M$ be a normal Noetherian scheme flat over $\Spec \ZZ_{(p)}$ such that $\modp{M}$ is $F$-finite.
  If $\modp{M}$ is quasi-$F$-split, then so is $M$.
\end{lemma}

\begin{proof}
  The natural map $\sO_M \to \sO_{\modp{M}}$ induces the diagram
  \[
    \begin{tikzcd}[row sep = small]
      \sO_{\modp{M}} \ar[r] \ar[rd] & Q_{M,n} \ar[d] \\
      & Q_{\modp{M}, n}
    \end{tikzcd}\text{.}
  \]
  The assertion follows from this diagram.
\end{proof}

\subsection{Global $F$-splitting via normalization}
Miller and Schwede gave a criterion for the $F$-purity of nonnormal rings via normalization \cite[Corollary 4.3]{MillerSchwede12}.
We present a global variant (\myref{prop:MillerSchwede-global}).
This may be known to the experts but the author was not able to find a suitable reference.

Recall that a Noetherian scheme $X$ is said to be \emph{\SerreG{1}} if it is Gorenstein at codimension $\le 1$ points.
See \cite{Hartshorne94:AC} for the definition of \emph{almost Cartier (AC) divisors}.

\begin{proposition}[{\cite[Theorem 2.4 and Remark 2.5]{MillerSchwede12}}]\label{prop:correspondence-divisor-splitting}
  Let $X$ be a reduced \SerreG{1} \SerreS{2} projective scheme over an $F$-finite field of characteristic $p > 0$. Then there is a natural bijection between the following two sets:
  \begin{gather*}
    \setbraces*{\text{effective $\ZZ_{(p)}$-AC divisors $\Delta$ such that $\Kk_X + \Delta$ is $\ZZ_{(p)}$-Cartier}} \\
    \updownarrow \\
    \setbraces*{
      \begin{gathered}
        \text{triples $(e > 0,\, \sL,\, \varphi \colon F^e\push \sL \to \sO_X)$ such that $\sL$ is a line bundle on $X$} \\
        \text{and $\varphi$ is nonzero at every generic point of $X$}
      \end{gathered}
    } \Bigg/ {\sim}\text{.}
  \end{gather*}
  Moreover, suppose that $\varphi \colon F^e\push \sL \to \sO_X$ corresponds to $\Delta \coloneqq 0$ by this bijection; then, the map $\varphi$ is surjective if and only if $X$ is globally $F$-split.
\end{proposition}

\begin{proof}
  When $X$ is normal, this is well known~\cite[Remark 9.5]{Schwede09:F-adjunction}.
  When $X = \Spec A$ is affine and $A$ is semi-local, the assertion is \cite[Theorem 2.4]{MillerSchwede12}.
  The general statement follows by combining these two arguments.
\end{proof}

\begin{proposition}[{see \cite[Corollary 4.3]{MillerSchwede12}}]\label{prop:MillerSchwede-global}
  Let $X$ be a $\ZZ_{(p)}$-Gorenstein reduced \SerreG{1} \SerreS{2} projective scheme over an $F$-finite field of characteristic $p > 0$.
  Assume that $X$ has hereditary surjective trace (see \myref{rmk:hereditary-surjective-trace} below).
  Let $\nu \colon Y \to X$ be the normalization of $X$.
  Then $X$ is globally $F$-split if and only if $(Y, \BY)$ is globally $F$-split.
\end{proposition}

\begin{remark}\label{rmk:hereditary-surjective-trace}
  In \myref{prop:MillerSchwede-global}, we assumed that $X$ has hereditary surjective trace.
  We only give a sufficient condition for it since we do not need the full generality of the definition in this paper.
  See \cite[Definition 3.5]{MillerSchwede12} for the general definition.
  Let
  \begin{equation*}
    \begin{tikzcd}
      \BY \ar[d, "\nu\restrict{\BY}"'] \ar[r, hook] & Y \ar[d, "\nu"] \\
      C \ar[r, hook] & X
    \end{tikzcd}
  \end{equation*}
  be the conductor square, and let $C\redn = \bigcup_\lambda C_\lambda$ be the decomposition into irreducible components.
  Suppose that all $C_\lambda$ are normal, and that $\Tr \colon \nu\push \sO_{\BY_\lambda\normaln} \to \sO_{C_\lambda}$ is surjective for every component $\BY_\lambda$ of $\BY\redn$ dominating $C_\lambda$.
  Then $X$ has hereditary surjective trace.
\end{remark}

\begin{proof}
  We note that $X$ is seminormal if $(Y, \BY)$ is globally $F$-split by \cite[Corollary 2.7 (i) and (vii)]{GrecoTraverso80}.
  Let $\varphi \colon F^e\push \sL \to \sO_X$ be the map that corresponds to $\Delta \coloneqq 0$ by \myref{prop:correspondence-divisor-splitting}.
  From \cite[Proposition 7.11]{Schwede10:Centers}, we deduce that $\varphi$ extends uniquely to a map $\varphi\normaln \colon F^e\push \sL_Y \to \sO_Y$, where $\sL_Y$ is the pullback of $\sL$ to $Y$.
  Let $\Delta_{\varphi\normaln}$ be the divisor that corresponds to $\varphi\normaln$ by \myref{prop:correspondence-divisor-splitting} (or~\cite[Remark 9.5]{Schwede09:F-adjunction}).
  By \cite[Lemma 4.2]{MillerSchwede12}, we obtain $\Delta_{\varphi\normaln} = \BY$. Note that although \cite[Lemma 4.2]{MillerSchwede12} is a statement for affine schemes, it can be extended to our situation by checking locally.
  Hence, $(Y,\BY)$ is globally $F$-split if and only if $\varphi\normaln$ is surjective. Also, $X$ is globally $F$-split if and only if $\varphi$ is surjective.
  We conclude from \cite[Theorem 3.10]{MillerSchwede12} that these two are equivalent, by noting that surjectivity can be checked locally.
\end{proof}

\section{Quasi-\texorpdfstring{$F$}{F}-splitting and global \texorpdfstring{$\bmplus$}{+}-regularity}\label{sec:qFs-gpr}
We show a sufficient condition for global $\bmplus$-regularity using quasi-$F$-splitting (\myref{thm:qFs-implies-gpr}).
\begin{notation}\label{nota:qfs-gpr}
  Let $(R,\frakm)$ be a complete DVR with $p \in \frakm$ and a normalized dualizing complex $\omega_R^\bullet$.
  Let $M$ be a $(d+1)$-dimensional integral normal scheme that is projective over $\Spec R$ for $d \geq 1$.
  We assume that $M$ is $p$-torsion free and $\modp{M}$ is $F$-finite.
  In this section, we write $\overline{M} \coloneqq \modp{M}$, and similarly for divisors and sheaves.
\end{notation}

\begin{lemma}\label{lem:nonnormal-Serre-cond}
  Let the notation be as in \myref{nota:qfs-gpr}.
  Let $\sL$ be a line bundle on $M$, and $i \in \ZZ$.
  If $\Hh^i_\frakm(\overline{M}, \sL\powotimes{p^e}|_{\overline{M}})=0$ for every $e \geq 1$, then 
  \[
    \Hh^i_\frakm(\overline{M}, Q_{M,n} \otimes \sL\powotimes{p^e})=0 \quad \text{for all } e \geq 0 \text{ and } n \geq 1.
  \]
\end{lemma}

\begin{proof}
  An argument analogous to \cite[Proposition~3.8]{OnukiTakamatsuYoshikawa25v2} applies in this case.
\end{proof}

\begin{proposition}\label{prop:nonnormal-dual-Kodaira-vanishing}
  Let the notation be as in \myref{nota:qfs-gpr}.
  Assume that $M$ is Cohen--Macaulay and quasi-$F$-split.
  Let $\sA$ be an ample line bundle on $M$.
  Then
  \[
    \Hh^j_\frakm(M,\sA^\inv) = 0
    \quad \text{for all } j<d+1.
  \]
\end{proposition}

\begin{proof}
  The proof is analogous to \cite[Theorem~3.9]{OnukiTakamatsuYoshikawa25v2}.
  We include a proof for the convenience of the reader.
  
  Fix $j < d$.
  First, we claim that $H^j_\frakm(\overline{M}, \cA^{-p^e}|_{\overline{M}}) = 0$ for every $e \gg 0$.
  This follows because, taking Matlis dual, we have
  \begin{equation*}
    \Hh^j_\frakm(\overline{M}, \sA^{-p^e})^\vee
    \isom \Ext^{d-j}_{\overline{M}}(\sA^{-p^e}, \omega_{\overline{M}})
    \isom \Hh^{d-j}(\overline{M}, \omega_{\overline{M}} \otimes \sA^{p^e}) = 0\text{,}
  \end{equation*}
  where the first isomorphism holds by Grothendieck and local duality (see~\cite[Lemma~2.2]{B+MMP}), by noting that $X$ is Cohen--Macaulay.
  By \myref{lem:nonnormal-Serre-cond}, $\Hh^j_\frakm(\overline{M}, Q_{M, n} \otimes \sA^{-p^e}\restrict{\overline{M}}) = 0$ for every $e \ge e_0 - 1$ and every $n \ge 1$.
  Since $M$ is $n$-quasi-$F$-split for some $n \ge 1$ by assumption, $\Hh^j_\frakm(\overline{M}, \sA^{-p^e}\restrict{\overline{M}}) = 0$ for every $e \ge e_0 - 1$.
  Repeating this, we obtain $\Hh^j_\frakm(\overline{M}, \sA^\inv\restrict{\overline{M}}) = 0$.

  Now consider the exact sequence
  \[
    0 \to \sA^\inv \xrightarrow{\cdot p} \sA^\inv \to \overline{\sA^\inv} \to 0.
  \]
  Taking the long exact sequence, we obtain $\Hh^j_\frakm(M, \sA^\inv) = 0$ for $j < d + 1$ since $p \in \frakm$ and the local cohomology modules are $\frakm$-torsion.
\end{proof}

\begin{theorem}\label{thm:qFs-implies-gpr}
  Let the notation be as in \myref{nota:qfs-gpr}.
  Assume that $M$ is quasi-$F$-split, $-\Kk_M$ is an ample Cartier divisor, and $M$ is perfectoid BCM-regular at every closed point. 
  Then $M$ is globally $\bmplus$-regular.
\end{theorem}
Here, see \cite[Definition 6.9]{MaSchwede21} for the definition of perfectoid BCM-regularity.

\begin{proof}
  Let $\tau(\sO_M) \subset \sO_M$ be the test ideal sheaf defined in~\cite{B+:Testideals}.
  Using the comparison~\cite[Theorem~B~(j)]{B+:Testideals}, we see $\tau(\sO_M) = \sO_M$ from the perfectoid BCM-regularity of $M$. From~\cite[Theorem~B~(j)]{B+:Testideals} and the ampleness of $-K_M$, it follows that there is some divisor $G \ge 0$ such that
  \[
    H^0(M, \tau(\sO_M) \otimes \omega_M^{-n}) = \mathbf{B}^0(M, \epsilon G; \omega_M^{-n}) \subset \mathbf{B}^0(M; \omega_M^{-n})
  \]
  for every $n \gg 0$ and $0 < \epsilon \ll 1$,
  where $\mathbf{B}^0(M, \epsilon G; \omega_M^{-n}) \subset H^0(M, \omega_M^{-n})$ is the submodule defined in~\cite[Definition~4.2]{B+MMP}.
  Therefore, we have
  \[
    H^0(M, \omega_M^{-n}) = H^0(M, \tau(\sO_M) \otimes \omega_M^{-n}) \subset \mathbf{B}^0(M; \omega_M^{-n}) \subset H^0(M, \omega_M^{-n})\text,
  \]
  that is, $\mathbf{B}^0(M; \omega_M^{-n}) = H^0(M, \omega_M^{-n})$.
  By~\cite[Lemma~4.8~(a)]{B+MMP}, there exists $e_0$ such that for every $e \geq e_0$, the map
  \[
    H^{d+1}_\frakm(M,\sO_M(p^eK_M)) \to H^{d+1}_\frakm(M^+,\sO_{M^+}(p^eK_M))
  \]
  is injective.

  We now show that
  \[
    H^{d+1}_\frakm(M,\sO_M(p^{e_0-1}K_M)) \to H^{d+1}_\frakm(M^+,\sO_{M^+}(p^{e_0-1}K_M))
  \]
  is also injective if $e_0 \geq 1$.
  By Bhatt's vanishing theorem \cite[Theorem~6.28]{Bhatt21v2:CM}, we have
    $H^d_\frakm(M^+, \sO_{M^+}(lK_M))=0$
  for every $l \geq 1$.
  Furthermore, by \myref{prop:nonnormal-dual-Kodaira-vanishing}, we also have
    $H^d_\frakm(M, \sO_{M}(lK_M))=0$
  for every $l \geq 1$.
  Thus we obtain the commutative diagram
  \begin{equation}\label{eq:+-reg}
    \begin{tikzcd}
      H^{d}_\frakm(\overline{M},\sO_{\overline{M}}(lK_{\overline{M}})) \arrow[r] \arrow[d,hookrightarrow] & H^{d}_\frakm(\overline{M^+},\overline{\sO_{M^+}(lK_M)}) \arrow[d,hookrightarrow] \\
      H^{d+1}_\frakm(M,\sO_M(lK_M)) \arrow[r] & H^{d+1}_\frakm(M^+,\sO_{M^+}(lK_M))
    \end{tikzcd} 
  \end{equation}
  for every $l \geq 1$.
  In particular, the map
  \[
    H^{d}_\frakm(\overline{M},\sO_{\overline{M}}(p^eK_{\overline{M}})) \to H^{d}_\frakm(\overline{M^+},\overline{\sO_{M^+}(p^eK_M)})
  \]
  is injective for all $e \geq e_0$.
  By \eqref{eq:exact-V-R}, we have the exact sequences
  \begin{gather*}
    0 \to F_*(Q_{M,n-1} \otimes \sO_M(p^{e_0}K_M)) \to Q_{M,n} \otimes \sO_M(p^{e_0-1}K_M) \to F_*\sO_{\overline{M}}(p^{e_0}K_{\overline{M}}) \to 0\text, \\
    0 \to F_*(Q_{M^+,n-1} \otimes \sO_{M^+}(p^{e_0}K_M)) \to Q_{M^+,n} \otimes \sO_{M^+}(p^{e_0-1}K_M) \to F_*\overline{\sO_{M^+}(p^{e_0}K_M)} \to 0
  \end{gather*}
  for every $n > 1$.
  Using Bhatt's vanishing theorem \cite[Theorem~6.28]{Bhatt21v2:CM}, we obtain $H^{d-1}_\frakm(\overline{M^+}, \overline{\sO_{M^+}(p^{e_0}K_M)}) = 0$.
  Taking the long exact sequences in local cohomology, we deduce that
  \[
    H^{d}_\frakm(\overline{M},Q_{M,n} \otimes \sO_M(p^{e_0-1}K_M)) \to H^{d}_\frakm(\overline{M^+},Q_{M^+,n} \otimes \sO_{M^+}(p^{e_0-1}K_M))
  \]
  is injective for every $n \geq 1$ by induction on $n$.
  Since $M$ is quasi-$F$-split, we conclude that
  \[
    H^{d}_\frakm(\overline{M},\sO_{\overline{M}}(p^{e_0-1}K_{\overline{M}})) \to H^{d}_\frakm(\overline{M^+},\overline{\sO_{M^+}(p^{e_0-1}K_M)})
  \]
  is injective.
  Therefore, by \eqref{eq:+-reg}, the map
  \[
    H^{d+1}_\frakm(M,\sO_M(p^{e_0-1}K_M)) \to H^{d+1}_\frakm(M^+,\sO_{M^+}(p^{e_0-1}K_M))
  \]
  is injective.
  Repeating this process, we eventually obtain that
  \[
    H^{d+1}_\frakm(M,\sO_M(K_M)) \to H^{d+1}_\frakm(M^+,\sO_{M^+}(K_M))
  \]
  is injective.
  Therefore, by~\cite[Proposition~6.8, Lemma~4.8~(a)]{B+MMP}, $M$ is globally $\bmplus$-regular.
\end{proof}

% \subsection{Fedder-type criterion}
% \rewrite{we review for convenience}

\section{Regular curves and smooth surfaces in mixed characteristic}\label{sec:curves-and-smooths}
We present classification results for regular Fano curves and smooth weak del Pezzo surfaces.
Most of the results in this section should be well known to the experts, but the author could not find suitable references.

\begin{remark}\label{rmk:multiple-fiber}
  Let $(R, \frakm, k)$ be a Noetherian local ring of mixed characteristic $(0, p)$.
  Let $M$ be a regular integral flat projective $R$-scheme with $\Hh^0(M, \sO_M) = R$.
  Suppose that $\omega_M^\inv$ is ample.
  Then $M_k$ is not a multiple fiber of $M \to \Spec R$: that is, there is no $a \ge 2$ such that $M_k = a D$ for some integral divisor $D$ on $M$.
  This follows from the well-known argument using the Euler characteristic $\chi(\sO_D)$ (see \cite[Proposition 1.4]{Fujita90}).
\end{remark}

\subsection{Regular Fano curves in mixed characteristic}
We classify regular Fano curves in mixed characteristic, based on the classical theory of arithmetic surfaces (see, for example, \cite{Liu02}).

\begin{proposition}[{\cite[Section 9]{Liu02}}]\label{prop:arithmetic-surface}
  Let $(R, \frakm, k)$ be a complete DVR of mixed characteristic $(0, p > 0)$ with $k$ algebraically closed.
  Let $M$ be a regular integral flat projective $R$-scheme with $\Hh^0(M, \sO_M) = R$.
  Assume $\dim M = 2$ and $\omega_M^\inv$ is big.
  Then there exists a sequence of morphisms
  \begin{equation*}
    \begin{tikzcd}
      M = M_0 \ar[r, "\mu_0"] & M_1 \ar[r, "\mu_1"] & \dotsm \ar[r, "\mu_{n - 1}"] & M_n \isom \PP^1_R
    \end{tikzcd}\text{,}
  \end{equation*}
  where each $\mu_i$ is the blowup of $M_{i + 1}$ at a closed point.
  In particular, $(M_k)\redn$ is an SNC divisor on $M$, and each of its components is $\PP^1_k$.
\end{proposition}

\begin{proof}
  The proof is by induction on the number of irreducible components of $M_k$.
  Suppose first that $M_k$ is irreducible. Then $M_k$ is reduced by \myref{rmk:multiple-fiber}.
  By \cite[Proposition 9.3.16 (a)]{Liu02}, we deduce that $M_k$ is a conic over $k$, noting that $k$ is algebraically closed. Hence, $M_k$ is isomorphic to $\PP^1_k$.
  Therefore, $M \isom \PP^1_R$ by \myref{lem:deform-projective-bundle} below.

  Next, suppose that $M_k$ is reducible. Let $\eta$ be the generic point of $\Spec R$. Then $M_\eta$ is smooth since $M$ is regular and $\residuefield{\eta}$ has characteristic zero.
  Since $\Hh^0(M, \sO_{M}) = R$ by assumption, we have $\Hh^0(M_\eta, \sO_{M_\eta}) = \residuefield{\eta}$.
  By \cite[Proposition 9.3.16 (a)]{Liu02}, $M_\eta$ is a conic over $\residuefield{\eta}$.
  Let $E$ be a prime divisor appearing in $M_k$. From $M_k \ndot E = 0$ and $M_k \equivlinearly 0$, we see $E\powndot{2} < 0$.
  We can choose $E$ so that $-\Kk_M \ndot E > 0$ since $-\Kk_M \ndot M_k = \deg(-\Kk_{M_\eta}) = 2$.
  By \cite[Proposition 9.3.10 (a) and Remark 9.3.2]{Liu02}, there exists a birational morphism $\mu_0 \colon M \to M_1$ to a regular integral flat projective $R$-scheme $M_1$ such that $\mu_0$ is the blowup of $M_1$ at $\mu_0(E)$.
  By the inductive hypothesis for $M_1$, we obtain the desired sequence of blowups.
\end{proof}

\begin{proposition}\label{prop:classification-of-2-dim-Fano}
  In the setting of \myref{prop:arithmetic-surface},
  assume that $\omega_M^\inv$ is ample.
  Write $\frakm = (t)$.
  Then $M$ is isomorphic to either $\PP^1_R$ or $M \isom (XY = tZ^2) \subset \PP^2_R$, where $X, Y, Z$ are homogeneous coordinates on $\PP^2_R$.
  In particular, $M$ is globally $\bmplus$-regular.
\end{proposition}

\begin{proof}
  Let $M_k = \sum_i a_i E_i$ ($a_i \ge 1$) be the irreducible decomposition.
  There exists $E_l$ such that $a_l = 1$.
  By the adjunction formula, we have
  \begin{equation*}
    -\Kk_{M}\restrict{E_l} \equivlinearly -\Kk_{E_l} + E_l\restrict{E_l}
    \equivlinearly 2H - \frac{1}{a_l}\sum_{i \in I_l} a_iE_i\text,
  \end{equation*}
  where $I_l = \setComprehension{i \neq l}{E_i \cap E_l \neq \emptyset}$ and $H$ is a divisor of degree $1$ on $E_l \isom \PP^1_k$.
  Since $a_l = 1$ and $-\Kk_{M}\restrict{E_l}$ is ample, we obtain $2 - \sum_{i \in I_l} a_i > 0$.
  Thus, $\card{I_l}$ is $0$ or $1$, and we deduce that $M \isom \PP^1_R$ or $M \isom \Bl_x \PP^1_R$ for a closed point $x \in \PP^1_R$.
  The latter is isomorphic to $(XY = tZ^2) \subset \PP^2_R$.
  The global $\bmplus$-regularity of $M$ follows, for example, by combining \myref{thm:qFs-implies-gpr} and the Fedder-type criterion for quasi-$F$-splitting~\cite[Theorem 4.13]{Yoshikawa25}.
\end{proof}

% \begin{proposition}
%   Suppose, in addition, that $\omega_M^\inv$ is nef and big.
%   Then $M$ is one of the following.
% \end{proposition}

% \begin{corollary}
%   Suppose, in addition, that $p > 3$ and $\omega_M^\inv$ is nef and big.
%   Then $M$ is \rewrite{globally $\bmplus$-regular?}
% \end{corollary}

% Globally $\bmplus$-regular schemes are conjectured to be of Fano type~\cite[Conjecture 6.17]{B+MMP}.
% The author does not know if this holds even in the two-dimensional case.

\subsection{Smooth weak del Pezzo surfaces in mixed characteristic}
We present a classification result for smooth weak del Pezzo surfaces over a complete DVR of mixed characteristic.
Recall that every weak del Pezzo surface over an algebraically closed field is either $\PP^1 \times \PP^1$, $\rationalruled{2}$, or a blowup of $\PP^2$ at $n \le 8$ points.
Smooth weak del Pezzo surfaces in mixed characteristic are lifts of these surfaces, and we classify them based on deformation theory.

\begin{lemma}\label{lem:isomorphic-closed-fibers}
  Let $(R, \frakm, k)$ be a Noetherian local ring, and let $f \colon M \to N$ be a morphism between flat projective $R$-schemes.
  If $f_k = f \otimes_R k \colon M_k \to N_k$ is an isomorphism, then $f$ is an isomorphism.
\end{lemma}

\begin{proof}
  Since $f_k$ is an isomorphism, we deduce that $f$ is quasi-finite; hence finite by Zariski's main theorem. Define $\sO_N$-modules $\mathcal{K}$, $\mathcal{C}$ by the exact sequence
  \[
    0 \to \mathcal{K} \to \sO_N \to f\push \sO_M \to \mathcal{C} \to 0\text{.}
  \]
  Since $f_k$ is an isomorphism, we obtain $\mathcal{C} = 0$ by Nakayama's lemma.
  Note that since $f_k$ is an isomorphism, $f$ is flat (by the fiberwise criterion for flatness \citestacks{039B} for example).
  Hence, we see
  \[
    0 \to \mathcal{K}_k \to \sO_{N_k} \to (f\push \sO_M)_k \to 0
  \]
  is exact. Since $f_k$ is an isomorphism, we obtain $\mathcal{K} = 0$ by Nakayama's lemma.
  Therefore, $f$ is an isomorphism by noting that $f$ is affine.
\end{proof}

\begin{lemma}\label{lem:deform-projective-bundle}
  Let $(R, \frakm, k)$ be a complete Noetherian local ring.
  Suppose that $M$ is a flat projective $R$-scheme and the closed fiber $M_k$ is isomorphic to $\PP_{N_0}(\sE_0)$ for a locally free sheaf $\sE_0$ on a projective $k$-scheme $N_0$.
  If $\Hh^2(M_k, \sO_{M_k}) = 0$, then there exist a flat projective $R$-scheme $N$ with $N_k \isom N_0$, a locally free sheaf $\sE$ on $N$ with $\sE_k \isom \sE_0$ (via $N_k \isom N_0$), and an isomorphism $M \isom \PP_N(\sE)$ that restricts to the given isomorphism $M_k \isom \PP_{N_0}(\sE_0)$.
\end{lemma}

\begin{proof}
  We note that, by \cite[Corollary 8.5.6]{FGAexplained05} and $\Hh^2(M_k, \sO_{M_k}) = 0$, any line bundle on $M_k$ has a lift to $M$.
  Take a sufficiently ample line bundle $\sO_{N_0}(1)$ on $N_0$ so that it is globally generated and $\Hh^1(N_0, \sO_{N_0}(i)) = 0$ for all $i \ge 1$.
  Set $\pi_0 \colon M_k \isom \PP_{N_0}(\sE_0) \to N_0$ to be the projection and set $\sL_0 \coloneqq \pi_0\pull \sO_{N_0}(1)$. Take a lift $\sL$ of $\sL_0$ to $M$.
  Since $\Hh^0(M, \sL) \otimes k \isom \Hh^0(M_k, \sL_0)$ by cohomology and base change, we see that $\sL$ is globally generated.
  Hence, we have a natural morphism
  \[
    \pi \colon M \to N \coloneqq \Proj \sectionRing{M}{\sL}\text.
  \]
  By cohomology and base change, we have $N_k \isom \Proj \sectionRing{M_k}{\sL_k} \isom N_0$.
  From the fiberwise criterion of flatness \citestacks{039B}, we deduce that $\pi$ is flat.
  Now, we take a lift $\sA$ of $\sO_{\PP_{N_0}(\sE_0)}(1)$ to $M$, and set $\sE \coloneqq \pi\push \sA$.
  By cohomology and base change, $\sE$ is a locally free sheaf on $N$. Furthermore, we claim that $\sE_k \isom \sE_0$ (through the identification $N_k \isom N_0$).
  For each closed point $y \in N$, we have the diagram
  \[
    \begin{tikzcd}[row sep = small]
      \sE \otimes \residuefield{y} \ar[r, "\isom"] & \Hh^0(\pi^\inv(y), \sA) \\
      \pi\push \sO_{\PP_{N_0}(\sE_0)}(1) \otimes \residuefield{y} \ar[u] \ar[r, "\isom"] & \Hh^0(\pi^\inv(y), \sO_{\PP_{N_0}(\sE_0)}(1)) \ar[u, "\isom" sloped]
    \end{tikzcd}
  \]
  by cohomology and base change. Hence, by Nakayama's lemma, we obtain the claim.
  Since $\sA$ is relatively ample over $N$, we have a natural morphism $\alpha \colon M \to \PP_N(\sE)$. Thus, $\alpha_k$ is an isomorphism, and hence $\alpha$ is also an isomorphism by \myref{lem:isomorphic-closed-fibers}.
\end{proof}

\begin{lemma}\label{lem:deform-blowup}
  Let $(R, \frakm, k)$ be a complete Noetherian local ring.
  Suppose that $M$ is a flat projective $R$-scheme and the closed fiber $M_k$ is isomorphic to the blowup $\Bl_{z_0}N_0$ of a projective $k$-scheme $N_0$ at a smooth point $z_0 \in N_0$.
  If $\Hh^2(M_k, \sO_{M_k}) = 0$, then there exist a flat projective $R$-scheme $N$ and a closed subscheme $Z$ of $N$ such that $Z \isom \Spec R$ and $M \isom \Bl_Z N$.
\end{lemma}

\begin{proof}
  We note that by \cite[Corollary 8.5.6]{FGAexplained05} and $\Hh^2(M_k, \sO_{M_k}) = 0$, any line bundle on $M_k$ has a lift to $M$.
  Write $\pi_0 \colon M_k \isom \Bl_{z_0}N_0 \to N_0$ for the natural morphism.
  Let $\sO_{N_0}(1)$ be a sufficiently ample line bundle on $N_0$, and set $\sL_0 \coloneqq \pi_0\pull \sO_{N_0}(1)$.
  Take a lift $\sL$ of $\sL_0$ to $M$.
  There exists an effective Cartier divisor $E$ on $M$ such that $E_k \isom E_0$ because $\Hh^1(E_0, \sO_{M_k}(E_0)\restrict{E_0}) = 0$.
  By cohomology and base change, we have the diagram:
  \[
    \begin{tikzcd}[row sep = small]
      M \ar[r, "\pi"] & N \coloneqq \Proj \sectionRing{M}{\sL} \\
      E \ar[u, hook] \ar[r] & Z \coloneqq \Proj \sectionRing{E}{\sL\restrict{E}} \ar[u, hook] \isom \Spec R
    \end{tikzcd}\text{.}
  \]
  By construction, we have a line bundle $\sO_N(1)$ on $N$ such that $\pi\pull \sO_N(1) \isom \sL$.
  Let $\sI_E$ and $\sI_Z$ denote the defining ideals of $E$ and $Z$, respectively. By cohomology and base change, we deduce that $\sI_E \otimes \sL$ is globally generated.
  Similarly, we see that $\sI_Z \otimes \sO_{N}(1)$ is globally generated by $\Hh^0(\sI_E \otimes \sL) \subset \Hh^0(\sL) \isom \Hh^0(\sO_N(1))$.
  Thus we have $\sI_E = \pi^\inv \sI_Z$.
  Therefore, we have a natural morphism $\nu \colon M \to \Bl_Z N$ by the universal property of blowups.
  Since $\nu_k$ is an isomorphism, so is $\nu$ by \myref{lem:isomorphic-closed-fibers}.
\end{proof}

\begin{example}\label{exa:deformation-F2}
  Let $(R, \frakm, k)$ be a complete DVR.
  Let $X$, $Y$ be homogeneous coordinates on $\PP^1_R$.
  For each $f \in R$, we define the locally free sheaf $\sE^f$ of rank $2$ on $\PP^1_R$ by the transition function:
  \[
    \begin{tikzcd}[ampersand replacement=\&, row sep = small]
        \sO^{\oplus 2}_{D_+(XY)}
        \ar[rr, "{\begin{pmatrix} 1 & 0 \\ Yf/X & (Y/X)^2 \end{pmatrix}}"]
      \& \& \sO^{\oplus 2}_{D_+(XY)} \\
      \& \ar[ul, "\isom" sloped, "\varphi_X\restrict{D_+(XY)}"] \sE^f\restrict{D_+(XY)} \ar[ur, "\isom" sloped, "\varphi_Y\restrict{D_+(XY)}"'] \&
    \end{tikzcd}\text{,}
  \]
  where we write $D_+(h) \coloneqq (h \neq 0) \subset \PP^1_R$ for a homogeneous polynomial $h \in R[X,Y]$, and $\varphi_X$ and $\varphi_Y$ are local trivializations on $D_+(X)$ and $D_+(Y)$, respectively.
  For example, we have $\sE^0 \isom \sO_{\PP^1_R} \oplus \sO_{\PP^1_R}(-2)$.
  If $f \in \frakm$, then $M \coloneqq \PP_{\PP^1_R}(\sE^f)$ is a flat projective $R$-scheme with closed fiber $M_k \isom \rationalruled{2}$.
  Moreover, we have $\sE^f \otimes_R R/\frakm^i \isom \sO_{\PP^1_{R/\frakm^i}} \oplus \sO_{\PP^1_{R/\frakm^i}}(-2)$ if $f \in \frakm^i$.
\end{example}

\begin{proposition}[{cf.~\cite[Propositions~3.7, 3.6]{Scholl85}}]\label{prop:smooth-wdP}
  Let $(R, \frakm, k)$ be a complete DVR with $k$ algebraically closed, and set $\frakm = (t)$.
  Let $M$ be an integral smooth projective $R$-scheme.
  Suppose that $\omega_M^\inv$ is nef and big.
  Then $M$ is isomorphic to one of the following $R$-schemes:
  \begin{itemize}
    \item an iterated blowup of $\PP^2_R$ along at most eight $R$-valued points, or
    \item $\PP^1_R \times \PP^1_R$, or
    \item $\PP_{\PP^1_R}(\sE^{0})$ or $\PP_{\PP^1_R}\left(\sE^{t^i}\right)$ for $i \ge 1$, which are defined in \myref{exa:deformation-F2}.
  \end{itemize}
\end{proposition}

\begin{proof}
  First, note that $M_k$ is a smooth weak del Pezzo surface since $-\Kk_{M_k}$ is nef and $(-\Kk_{M_k})\powndot{2} > 0$.

  If $M_k$ is isomorphic to $\PP^2_k$, then by \myref{lem:deform-projective-bundle}, $M$ itself is isomorphic to $\PP^2_R$.
  Next, suppose that $M_k$ is an iterated blowup of $\PP^2_k$.
  Then, by \myref{lem:deform-blowup}, $M$ itself is isomorphic to an iterated blowup of $\PP^2_R$ at $R$-valued points.

  Now consider the case where $M_k$ is isomorphic to $\rationalruled{2}$.
  We have $M \isom \PP_{\PP^1_R}(\sE)$ for some locally free sheaf $\sE$ such that $\sE_k \isom \sO_{\PP^1_R} \oplus \sO_{\PP^1_R}(-2)$ by \myref{lem:deform-projective-bundle}.
  We claim that $\sE \isom \sE^f$ for some $f \in \frakm$.
  By Grothendieck's existence theorem \cite[Theorem 8.4.2]{FGAexplained05} (together with \cite[8.1.4]{FGAexplained05}), the locally free sheaves on $\PP^1_R$ correspond bijectively with the compatible systems of locally free sheaves $E_n$ on $\PP^1_{R_n} = \PP^1 \otimes R_n$, where we set $R_n \coloneqq R/\frakm^{n+1}$ for $n \ge 0$.
  Fix $n \ge 0$ and assume $E_n$ is of the form $\sE^{f_n} \otimes R_n$. We show that any lift $E_{n + 1}$ of $E_n$ is also of the form $\sE^{f_{n+1}} \otimes R_{n + 1}$.
  Set $I_{n+1} \coloneqq \frakm^{n+1} / \frakm^{n+2} \subset R_{n + 1}$.
  By \cite[Theorem 8.5.3]{FGAexplained05}, the set of deformations
  \[
    \setbraces*{(E_{n+1}, \alpha \colon E_{n+1} \to E_n) \text{ such that $\alpha \otimes R_n$ is an isomorphism}}
    / {\isom}
  \]
  is an affine space under $\Hh^1(M \otimes R_n, I_{n+1} \otimes \sEnd(E_n)) \isom I_{n+1} \isom k$. The action of $g \in I_{n+1}$ can be computed via \v{C}ech cohomology, and we have
  \[
    g \cdot [\sE^{f_n} \otimes R_{n + 1}] \isom [\sE^{f_n + g} \otimes R_{n + 1}]\text{.}
  \]
  Thus, any lift $E_{n + 1}$ of $E_n$ is also of the form $\sE^{f_{n+1}} \otimes R_{n + 1}$. This proves the case $M_k \isom \rationalruled{2}$.
  The case of $\PP^1_k \times \PP^1_k$ is similar and easier.
\end{proof}

% Assume $R = \W(k)$. Then it is shown in \cite[Theorem~4.10]{OnukiTakamatsuYoshikawa25v2} that $M$ is quasi-$F$-split.

\section{Gorenstein del Pezzo surfaces over fields}\label{sec:qFs-dP-over-fields}
We describe which Gorenstein del Pezzo surfaces are quasi-$F$-split.
Throughout this section, fix an algebraically closed field $k$ of characteristic $p>0$.

First, we recall the situation for RDP del Pezzo surfaces.% Although this is not explicitly used in the paper, it is essential in~\cite{OnukiTakamatsuYoshikawa25v2}, which is crucial for \myref{main:globally-plus-regular}.
\begin{definition}\label{def:log-lift}
  An RDP del Pezzo surface $X$ over $k$ is \emph{$\W(k)$-log liftable} if there exists a log resolution $(Y, E_0) \rightarrow X$ and a projective log-smooth pair $(N, E)$ over $\W(k)$ such that $(N_k, E_k) \simeq (Y, E_0)$.
  Here, a projective \emph{log-smooth pair} is, by definition, a pair $(N, E)$ where $N$ is a smooth projective scheme over $\W(k)$ and $E \subset N$ is a relatively SNC divisor.
\end{definition}

\begin{remark}
  Let $X$ be an RDP del Pezzo surface. Then $X$ is quasi-$F$-split if and only if $X$ is $\W(k)$-log liftable by~\cite[Theorem 6.3]{KTTWYY:qFs1}.
  Non-$\W(k)$-log liftable RDP del Pezzo surfaces are classified by~\cite[Theorem 1.7 (1)]{KawakamiNagaoka22:Pathologies} and~\cite[Propositions 3.14, 3.15, and 3.23]{KawakamiNagaoka23:Classification}.
  Thus, non-quasi-$F$-split RDP del Pezzo surfaces are explicitly described, and this observation plays an crucial role in the proof of \cite[Theorem~B]{OnukiTakamatsuYoshikawa25v2}.
\end{remark}

Next, we treat normal but non-RDP Gorenstein del Pezzo surfaces.
\begin{lemma}\label{lem:qFs-elliptic-cone}
  Let $X$ be a normal Gorenstein del Pezzo surface that is not an RDP del Pezzo surface. Then $X$ is quasi-$F$-split.
\end{lemma}

\begin{proof}
  By \myref{thm:normal-dP-classification}, $\sectionRing{X}{\omega_X^\inv} \isom \sectionRing{E}{\sA}[t]$ for an ample line bundle $\sA$ on an elliptic curve $E$.
  $E$ is quasi-$F$-split by~\cite[Corollary 7.5]{Silverman09:AEC} together with the characterization of quasi-$F$-splitting via formal groups~\cite[Theorem 4.5]{Yobuko19}.
  From~\cite[Theorem 7.16]{KTTWYY:qFs2} combined with~\cite[Proposition 2.11]{KawakamiTakamatsuYoshikawa25:Fedder}, we see that $\sectionRing{E}{\sA}$ is quasi-$F$-split.
  By~\cite[Theorem 1.3]{Yobuko23}, it follows that $\sectionRing{E}{\sA}[t]$ is quasi-$F$-split.
  Using~\cite[Theorem 7.16]{KTTWYY:qFs2} and~\cite[Proposition 2.11]{KawakamiTakamatsuYoshikawa25:Fedder} once again, we conclude that $X$ is quasi-$F$-split.
\end{proof}

We now study nonnormal Gorenstein del Pezzo surfaces.

\begin{lemma}\label{lem:toricity-of-normalization}
  Let $X$ be a nonnormal Gorenstein del Pezzo surface and let $(Y_i, \BY_i)$ be a connected component of the normalization of $X$.
  Then one of the following holds:
  \begin{enumerate}
    \item $Y_i$ is a toric variety and $\BY_i$ is a toric divisor on $Y_i$, or
    \item $Y_i = \PP^2$ and $\BY_i$ is a smooth conic in $\PP^2$, or
    \item $Y_i = \PP^1 \times \PP^1$ and $\BY_i$ is a smooth curve in $\abs{\sO_{\PP^1 \times \PP^1}(1, 1)}$.
  \end{enumerate}
\end{lemma}

\begin{proof}
  This follows by inspecting \myref{thm:normalization-dP-classification}.
\end{proof}

\begin{corollary}\label{cor:gFs-of-normalization}
  Let $(Y_i, \BY_i)$ be as in \myref{lem:toricity-of-normalization}.
  Then $(Y_i, \BY_i)$ is globally $F$-split unless $\BY_i$ is a double line.
\end{corollary}

\begin{proof}
  By assumption, $\BY_i$ is a reduced curve.
  Hence, in the case of (1), the assertion is well known (see \cite[Corollary 2.5]{HaraWatanabe02}).
  In the cases of (2) and (3), the assertion follows from a direct computation using Fedder's criteria.
\end{proof}

\begin{proposition}\label{prop:gFs-of-nonnormal}
  Let $X$ be a tame nonnormal Gorenstein del Pezzo surface.
  We use the notation of \myref{nota:nonnormal-dP}: let $\nu \colon Y \to X$ be the normalization of $X$, and set $C \subset X$ and $\BY \subset Y$ to be the closed subschemes defined by the conductor ideals.
  Then $X$ is globally $F$-split, except in the following cases:
  \begin{enumerate}[label=(\thetheorem.\arabic*), ref=\thetheorem.\arabic*]
    \item\label{item:double-line-case} All connected components of $\BY$ are double lines.
    \item\label{item:inseparable-case} $X$ is irreducible, $\BY$ is a smooth conic, $p = 2$ and $\nu\restrict{\BY} \colon \BY \to C$ is an inseparable double cover.
  \end{enumerate}
\end{proposition}

\begin{proof}
  Assume the negations of \myref{item:double-line-case} and \myref{item:inseparable-case}.
  Since \myref{item:inseparable-case} does not hold, from \myref{thm:irreducible-dP-classification} and \ref{thm:reducible-dP-classification}, we see that $X$ has hereditary surjective trace. Thus we can apply \myref{prop:MillerSchwede-global}.
  Since \myref{item:double-line-case} does not hold, \myref{cor:gFs-of-normalization} shows that $(Y, B)$ is globally $F$-split, and hence $X$ is globally $F$-split.
\end{proof}

Although the following proposition is not needed later, we include it here for completeness.

\begin{proposition}\label{prop:qFs-of-nonnormal-dP}
  Let $X$ be a nonnormal Gorenstein del Pezzo surface.
  \begin{enumerate}
    \item $X$ is not seminormal if and only if $X$ is not tame or \myref{item:double-line-case} holds.
    \item $X$ is seminormal but not weakly normal if and only if \myref{item:inseparable-case} holds.
    \item $X$ is weakly normal if and only if it is globally $F$-split.
  \end{enumerate}
  In particular, $X$ is quasi-$F$-split if and only if it is globally $F$-split.
\end{proposition}

\begin{proof}
  Let $\nu \colon Y \to X$ be the normalization of $X$ and let $\BY \subset Y$ be the closed subscheme defined by the conductor ideal.
  By~\cite[Corollary 2.7 (i) and (vii)]{GrecoTraverso80}, $X$ is seminormal if and only if $\BY$ is reduced. Thus (1) follows \myref{def:tame} and \myref{thm:normalization-dP-classification}.
  If \myref{item:inseparable-case} holds, then $\nu$ is weakly subintegral and hence $X$ is not weakly normal. This proves the ``if'' part of (2). The ``if'' part of (3) follows from \myref{lem:qFs-implies-weakly-normal}.
  Thus \myref{prop:gFs-of-nonnormal} proves the claim.
\end{proof}

\section{Regular del Pezzo surfaces in mixed characteristic}\label{sec:main}
In the setting of \myref{main:globally-plus-regular}, we present a classification of the closed fiber $M_k$ of $M$, which is inspired by~\cite{Fujita90}.
As a consequence, we establish \myref{main:globally-plus-regular}.

\begin{setting}\label{set:main-theorem}
  Let $(R,\frakm_R,k)$ be a complete DVR of mixed characteristic $(0,p>0)$ with algebraically closed residue field $k$.
  Let $M$ be a regular integral flat projective $R$-scheme with $\Hh^0(M, \sO_M) = R$.
  Suppose that $\dim M = 3$ and $\omega_M^\inv$ is ample.
\end{setting}

\begin{lemma}
  Assume that $M_k$ is reduced.
  Then $M_k$ is a tame Gorenstein del Pezzo surface.
\end{lemma}

\begin{proof}
  Let $\eta$ be the generic point of $\Spec R$.
  We have $\chi(\sO_{M_k}) = \chi(\sO_{M_\eta}) = 1$ since $M_\eta$ is a regular del Pezzo surface over a field $\residuefield{\eta}$ of characteristic zero.
  Since $M_k$ is a Gorenstein del Pezzo surface, we obtain $\Hh^2(M_k,\sO_{M_k}) = 0$ by Serre duality.
  We have $\Hh^0(M_k,\sO_{M_k}) = k$ since $M_k$ is reduced.
  Therefore, we see $\Hh^1(M_k,\sO_{M_k}) = 0$, as desired.
\end{proof}

\subsection{Nonnormal irreducible case}
We treat the case where the closed fiber is nonnormal and irreducible.

\begin{lemma}\label{lem:factor-normalization}
  Let $X$ be a reduced projective $k$-scheme, and let $C \subset X$ be the closed subscheme defined by the conductor ideal.
  Write $\nu \colon Y \to X$ for the normalization of $X$.
  Then the blowup $\mu \colon \widetilde{X} \to X$ along $C$ factors through the normalization:
  \begin{equation*}
    \begin{tikzcd}
      \widetilde{X} \ar[r, dashed, "\exists"'] \ar[rr, bend left, "\mu" near start] & Y \ar[r, "\nu"'] & X
    \end{tikzcd}\text{.}
  \end{equation*}
\end{lemma}

\begin{proof}
  Write $\sI$ for the conductor ideal of $X$.
  We have
  \begin{equation*}
    \widetilde{X} \isom \sProj_X \bigoplus_{i \ge 0} \sI^i
    \isom \sProj_X \left(\nu\push \sO_Y \oplus \sI \oplus \sI^2 \oplus \cdots\right)\text{.}
  \end{equation*}
  Then the right-hand side is naturally a $Y$-scheme since $\nu$ is an affine morphism and $\sI^i$ is a $\nu\push \sO_Y$-module.
  This gives the assertion.
\end{proof}

\begin{lemma}\label{lem:intersection-number-by-classification}
  Let $X$ be a tame nonnormal irreducible Gorenstein del Pezzo surface.
  We use the notation of \myref{nota:nonnormal-dP}: write $d \coloneqq (-\Kk_X)\powndot{2}$ and $Y \coloneqq X\normaln$, and let $C \subset X$ be the subscheme defined by the conductor ideal.
  Let $\widetilde{X}$ be the blowup of $X$ along $C$, and let $\mu^\inv C \subset \widetilde{X}$ be the pullback of $C$ as a closed subscheme.
  % Let $C$ be the closed subscheme of $X$ defined by the conductor ideal.
  % Write $d \coloneqq (-\Kk_X)\powndot{2}$.
  % Let $Y$ be the normalization of $X$ and let $\widetilde{X}$ be the blowup of $X$ along $C$.
  % Suppose $\mu^\inv C \subset \widetilde{X}$ is the pullback of $C$ as a closed subscheme, which is a Cartier divisor.
  Then
  \begin{equation*}
    (\mu^\inv C)_{\widetilde{X}}\powndot{2} = \begin{cases}
      4 - d & (\text{if $Y \not\isom \PP^2$}) \\
      5 - d & (\text{if $Y \isom \PP^2$})\text{.}
    \end{cases}
  \end{equation*}
  Note that $Y \isom \PP^2$ holds only when $d = 1$ or $4$ by \myref{thm:normalization-dP-classification}.
\end{lemma}

\begin{proof}
  By \myref{lem:factor-normalization}, $\widetilde{X}$ is also the blowup of $Y$ along $\BY$.
  We show the assertion by a direct computation based on \myref{thm:normalization-dP-classification}.
  Assume that $\BY$ is Cartier. This holds precisely when $Y$ is of case (c) with $d = 2$, or of case (a), (b), (d) or (e) from \myref{tab:normalization}. Then we have $\widetilde{X} = Y$, and $(\mu^\inv C)_{\widetilde{X}}\powndot{2} = (\BY)_Y\powndot{2}$ can be computed case by case.
  We turn to the other case: assume that $Y = \PP(1,1,d)$ and $d \ge 3$. Then $\widetilde{X} \isom \rationalruled{d}$ and $\mu^\inv C \equivlinearly 2F + E$. Hence, $(\mu^\inv C)_{\widetilde{X}}\powndot{2} = 4 - d$ in this case.
\end{proof}

\begin{lemma}[{cf.~proof of \cite[Theorem 3.1]{Fujita90}}]\label{lem:intersection-number-by-blowup-M}
  We work under \myref{set:main-theorem}.
  Let $\mu_M \colon \widetilde{M} \to M$ be the blowup of $M$ along $C \subset X \subset M$ and let $\mu_X \colon \widetilde{X} \to X$ be its restriction to the strict transform $\widetilde{X}$ of $X$.
  Write $\mu_X^\inv C \subset \widetilde{X}$ for the pullback of $C$ as a closed subscheme of $X$.
  Then $(\mu_X^\inv C)_{\widetilde{X}}\powndot{2} = -2$.
\end{lemma}

\begin{proof}
  Let $E \subset \widetilde{M}$ be the exceptional divisor of the blowup $\mu_M$ and write $\mu_E \colon E \to C$ for the restriction of $\mu_M$ to $E$.
  We have $C \isom \PP^1_k$ by \myref{thm:irreducible-dP-classification}.
  Since $C$ and $M$ are regular, we see $\mu_E \colon E \isom \PP_{C}(\conormal{C}{M}) \to C$ is a projective bundle.
  Note that $-\Kk_M\restrict{X} \equivlinearly -\Kk_X$ by the adjunction formula and $-\Kk_M \ndot C = -\Kk_X \ndot C = 1$ by \myref{thm:irreducible-dP-classification}.
  Since $X = \divisor_X p \equivlinearly 0$, we have
  \begin{equation*}
    \mu_M\pull X \equivlinearly \widetilde{X} + a E \equivlinearly 0
  \end{equation*}
  for some $a \in \ZZ$.
  We see
  \begin{equation*}
    \mu_M\pull(-\Kk_M) \ndot E \ndot E
    = \mu_E\pull(-\Kk_M\restrict{C}) \ndot E\restrict{E}
    = -\deg(-\Kk_M\restrict{C}) = -1
  \end{equation*}
  and
  \begin{align*}
    \mu_M\pull(-\Kk_M) \ndot E \ndot \widetilde{X}
    &= \mu_X\pull(-\Kk_M\restrict{X}) \ndot E\restrict{\widetilde{X}}
    = \mu_X\pull(-\Kk_X) \ndot \mu_X^\inv C \\
    &= \nu\pull(-\Kk_X) \ndot \BY = 2
  \end{align*}
  by noting \myref{rmk:Reid-table}.
  Therefore, we obtain
  \[
    2 = \mu_M\pull(-\Kk_M) \ndot \widetilde{X} \ndot E = -a \, \mu_M\pull(-\Kk_M) \ndot E \ndot E = a\text{,}
  \]
  that is, $a = 2$.
  Then we deduce
  \begin{equation*}
    (\mu_X^\inv C)_{\widetilde{X}}\powndot{2}
    = E \ndot E \ndot \widetilde{X}
    = -2 \, E\powndot{3} = -2\deg \conormal{C}{M} = -2\text{,}
  \end{equation*}
  which is our claim.
\end{proof}

\begin{theorem}[{cf.~\cite[Subsection 6.1]{Fukuoka20:dP6}}]\label{thm:irreducible-main}
  Suppose that $k$ and $M$ are as in \myref{set:main-theorem}.
  Assume $X \coloneqq M_k$ is nonnormal and irreducible.
  Then we have $d = 6$. More precisely,
  \begin{equation*}
    (Y, \BY) \isom (\rationalruled{2}, E) \text{ or } (\rationalruled{4}, F + E)\text{.}
  \end{equation*}
\end{theorem}

\begin{proof}
  By \myref{lem:intersection-number-by-classification} and \myref{lem:intersection-number-by-blowup-M}, we obtain $d = 6$.
  We note that the embedding dimension of $X$ at any point is at most~$3$ since $M$ is regular of dimension~$3$.
  Thus, we get the assertion from \myref{exa:Reid-irreducible}.
\end{proof}

\begin{remark}\label{rmk:dP6-exist}
  We expect that there are regular del Pezzo surfaces $M$ in mixed characteristic with these nonnormal irreducible closed fibers.
  Note that analogous examples in characteristic zero are constructed~\cite{Fukuoka20:dP6}.
\end{remark}

\begin{corollary}\label{cor:gFs-irreducible}
  In the setting of \myref{thm:irreducible-main}, if moreover $p > 2$, then $X$ is globally $F$-split.
\end{corollary}

\begin{proof}
  This is a consequence of \myref{thm:irreducible-main} and \myref{prop:gFs-of-nonnormal}.
\end{proof}

\subsection{Reducible case}
We now turn to the case of reducible closed fibers.
Such fibers are classified (assuming reducedness), and the result of the classification is analogous to Fujita's classification of reducible fibers of del Pezzo fibrations in characteristic zero~\cite{Fujita90}.

Our proof is inspired by that of~\cite[Theorem 5.2]{Loginov22}, and is based on the following simple observation.

\begin{remark}[{cf.~\cite[Corollary 3.8 Triple point formula]{Loginov22} and~\cite[Lemma 2.1]{Kulikov77}}]\label{rmk:triple-points-formula}
  We work under \myref{set:main-theorem}.
  Assume $X$ is reduced.
  Let $X = \sum_i X_i$ be the decomposition into irreducible components.
  Let $\Gamma \subset X$ be a curve. Then, since $X \equivlinearly 0$, we have
  \begin{equation*}
    \sum_i X_i \ndot \Gamma = 0\text{.}
  \end{equation*}
  If $\Gamma \subset X_l$ and $i \neq l$, then we have
  \begin{equation*}
    X_i \ndot \Gamma = (X_i \cap X_l \ndot \Gamma)_{X_l}\text{.}
  \end{equation*}
\end{remark}

\begin{lemma}\label{lem:intersection-description}
  Let $X$ be a reducible tame Gorenstein del Pezzo surface and let $X = \bigcup_{i=1}^r X_i$ be the irreducible decomposition.
  We use the notation of \myref{nota:nonnormal-dP}.
  Then $X_i \cap X_l$ for $i \neq l$ is as follows:
  \begin{itemize}
    \item If all $\BY_i$ are smooth conics, then $X_i \cap X_l = C$. Recall that $r = 2$ in this case.
    \item If all $\BY_i$ are line pairs, then $X_i \cap X_{i+1}$ is a line pair if $r = 2$ and a line if $r \ge 3$. $X_i \cap X_l$ is a point otherwise. Here, we use the notation of \myref{exa:Reid-reducible}.
    \item If all $\BY_i$ are double lines, then $X_i \cap X_l$ is a double line if $r = 2$ or a line if $r \ge 3$.
  \end{itemize}
\end{lemma}

\begin{proof}
  This follows by inspecting each case of \myref{exa:Reid-reducible}.
\end{proof}

\begin{theorem}\label{thm:reducible-main}
  We work under \myref{set:main-theorem}.
  Assume $X \coloneqq M_k$ is reducible and reduced.
  Then there are precisely eight possibilities for $X$ as given in \myref{tab:reducible}.
  In particular, we have $d = 8$ or $9$.
  \begin{table}[h]
    \centering
    \caption{Reducible fibers of regular del Pezzo surfaces.}
    \label{tab:reducible}
    \begin{tblr}{clll}
      $r$ & $(Y_1, \sO_{Y_1}(\BY_1), \text{$\BY_1$})$ & $(Y_2, \sO_{Y_2}(\BY_2), \text{$\BY_2$})$ & $(Y_3, \sO_{Y_3}(\BY_3), \text{$\BY_3$})$ \\ \hline\hline
      $3$ & $(\rationalruled{1}, \sO(F+E), \text{line pair})$ & the same as $Y_1$ & the same as $Y_1$ \\
      $2$ & $(\rationalruled{2}, \sO(F+E), \text{line pair})$ & the same as $Y_1$ & \\
      $2$ & $(\rationalruled{4}, \sO(E), \text{smooth conic})$ & $(\PP^2, \sO(2), \text{smooth conic})$ \\
      $2$ & $(\rationalruled{1}, \sO(E), \text{smooth conic})$ & $(\PP^2, \sO(1), \text{smooth conic})$ \\
      $2$ & $(\rationalruled{2}, \sO(E), \text{smooth conic})$ & $(\PP(1,1,2), \sO(2), \text{smooth conic})$ \\
      $2$ & $(\rationalruled{2}, \sO(E), \text{smooth conic})$ & $(\rationalruled{0}, \sO(F+E), \text{smooth conic})$ \\
      $2$ & $(\rationalruled{1}, \sO(E), \text{smooth conic})$ & $(\rationalruled{1}, \sO(F+E), \text{smooth conic})$ \\
      $2$ & $(\rationalruled{0}, \sO(E), \text{smooth conic})$ & the same as $Y_1$
    \end{tblr}
  \end{table}
\end{theorem}

\begin{proof}
  We use \myref{rmk:triple-points-formula} and \myref{thm:reducible-dP-classification}.
  Note that by \myref{rmk:remark-on-reducible}, each $Y_i$ is isomorphic to $X_i$ and thus we identify them.

  By \myref{thm:reducible-dP-classification}, all the $\BY_i$ are isomorphic. There are three cases: each $\BY_i$ is (1) a smooth conic, (2) a line pair, or (3) a double line.

  (1) Suppose that each $\BY_i$ is a smooth conic. Then $r = 2$ by \myref{thm:reducible-dP-classification} and $X_1 \cap X_2 = \BY$ by \myref{lem:intersection-description}.
  Setting $\Gamma = \BY$ in \myref{rmk:triple-points-formula}, we obtain
  \begin{equation*}
    0 = (X_1 \cap X_2 \ndot \BY)_{X_2} + (X_2 \cap X_1 \ndot \BY)_{X_1}
    = (\BY)_{X_1}\powndot{2} + (\BY)_{X_2}\powndot{2}\text{.}
  \end{equation*}
  Swapping $X_1$ and $X_2$ if necessary, we can assume that $(\BY)_{X_1}\powndot{2} \le 0$. Then by \myref{thm:normalization-dP-classification}, $Y_1 \isom \rationalruled{d_1 - 4}$ and $\sO_{Y_1}(\BY) \isom \sO_{\rationalruled{d_1 - 4}}(E)$ for some $d_1 \ge 4$.
  We classify $Y_2$ based on \myref{tab:normalization}.
  Assume that $Y_2$ is of case (d). In this case, $d_2 = 2$ or $3$, and we have $0 = (\BY)_{X_1}\powndot{2} + (\BY)_{X_2}\powndot{2} = -(d_1 - 4) + (4 - d_2)$. Hence, $(d_1, d_2) = (6, 2)$ or $(5, 3)$, and $(Y_1, Y_2) \isom (\rationalruled{2}, \rationalruled{0})$ or $(\rationalruled{1}, \rationalruled{1})$.
  The proofs for the other cases of $Y_2$ are similar.

  (2) Suppose that each $\BY_i$ is a line pair. Since every $X_i \cap X_l$ for $i \neq l$ is an effective Cartier divisor on $X_i$, we see $r \le 3$ from \myref{lem:intersection-description}.
  First, assume $r = 2$. Write $X_1 \cap X_2 = \Gamma_1 \cup \Gamma_2$.
  Setting $\Gamma = \Gamma_1$ in \myref{rmk:triple-points-formula}, we obtain
  \begin{align*}
    0 &= (\Gamma_1 + \Gamma_2 \ndot \Gamma_1)_{X_2} + (\Gamma_1 + \Gamma_2 \ndot \Gamma_1)_{X_1} \\
    &= (\Gamma_1)_{X_1}\powndot{2} + (\Gamma_1)_{X_2}\powndot{2} + (\Gamma_1 \ndot \Gamma_2)_{X_1} + (\Gamma_1 \ndot \Gamma_2)_{X_2} \\
    &> (\Gamma_1)_{X_1}\powndot{2} + (\Gamma_1)_{X_2}\powndot{2}\text{.}
  \end{align*}
  Thus we may assume that $(\Gamma_1)_{X_1}\powndot{2} < 0$. From \myref{thm:normalization-dP-classification}, we see $Y_1 \isom \rationalruled{d_1 - 2}$ and $\Gamma_1 = E$ for some $d_1 \ge 2$. Hence, $\Gamma_2$ is a fiber of $\rationalruled{d_1 - 2}$ and $(\Gamma_2)_{X_1}\powndot{2} = 0$. Setting $\Gamma = \Gamma_2$ in \myref{rmk:triple-points-formula}, we obtain
  \begin{equation*}
    0 = 1 + (\Gamma_2)_{X_2}\powndot{2} + (\Gamma_1 \ndot \Gamma_2)_{X_2}
    > (\Gamma_2)_{X_2}\powndot{2}\text{.}
  \end{equation*}
  Thus $Y_2 \isom \rationalruled{d_2 - 2}$ and $\Gamma_2 = E$ for some $d_2 \ge 2$.
  Since $(\Gamma_1 \ndot \Gamma_2)_{X_1} = (\Gamma_1 \ndot \Gamma_2)_{X_2} = 1$, we deduce that $d_1 = d_2 = 4$ and $Y_1 \isom Y_2 \isom \rationalruled{2}$.
  Then, assume $r = 3$. Set $\Gamma_{il}$ to be the line $X_i \cap X_l$ for $i \neq l$. By \myref{rmk:triple-points-formula}, we see that
  \begin{equation*}
    0 = (X_1 \ndot \Gamma_{12}) + (X_2 \ndot \Gamma_{12}) + (X_3 \ndot \Gamma_{12})
    > (\Gamma_{12})_{X_1}\powndot{2} + (\Gamma_{12})_{X_2}\powndot{2}
  \end{equation*}
  Hence, we can assume that $Y_2 \isom \rationalruled{d_2 - 2}$ and $\Gamma_{12} = E$ for some $d_2$. Then $(\Gamma_{23})_{X_2}\powndot{2} = 0$ and $(\Gamma_{12} \ndot \Gamma_{23})_{X_2} = 1$. \myref{rmk:triple-points-formula} implies
  \begin{equation*}
    0 = (\Gamma_{23})_{X_3}\powndot{2} + 1\text{.}
  \end{equation*}
  It follows that $Y_3 \isom \rationalruled{1}$ and $\Gamma_{23} = E$.
  Repeating this argument, we deduce that $Y_1 \isom \rationalruled{1}$ and $Y_2 \isom \rationalruled{1}$.

  (3) Finally, suppose that each $\BY_i$ is a double line.
  Write $\Gamma$ for the reduced line $C\redn$ on $X$.
  From \myref{thm:normalization-dP-classification}, we see that $(\Gamma)_{Y_i}\powndot{2} > 0$ for every $i$.
  If $r = 2$, then by \myref{lem:intersection-description} and \myref{rmk:triple-points-formula}, we have
  \begin{equation*}
    0 = (\BY_2 \ndot \Gamma)_{X_2} + (\BY_1 \ndot \Gamma)_{X_1}
    = 2(\Gamma)_{X_1}\powndot{2} + 2(\Gamma)_{X_2}\powndot{2}\text{,}
  \end{equation*}
  which is a contradiction.
  If $r \ge 3$, then by a similar reasoning, we have
  \begin{equation*}
    0 = \sum_{1 \le i \le r} (\Gamma)_{X_{l(i)}}\powndot{2}\text{,}
  \end{equation*}
  where we chose $l(i) \neq i$ for each $i$.
  This is also a contradiction.
\end{proof}

\begin{corollary}\label{cor:gFs-reducible}
  In the setting of \myref{thm:reducible-main}, $X$ is globally $F$-split.
\end{corollary}

\begin{proof}
  This is a consequence of \myref{thm:reducible-main} and \myref{prop:gFs-of-nonnormal}.
\end{proof}

\subsection{Proofs of \myref{main:globally-plus-regular} and \myref{main:vanishing}}
\begin{proof}[\proofname{} of \myref{main:globally-plus-regular}]
  By \myref{thm:qFs-implies-gpr}, it suffices to show that $M$ is quasi-$F$-split.
  If $X$ is an RDP del Pezzo surface, $M$ is quasi-$F$-split by \cite[Theorem~B]{OnukiTakamatsuYoshikawa25v2}.
  When $X$ is not an RDP del Pezzo surface, we show that $X$ is quasi-$F$-split: this implies that $M$ is quasi-$F$-split by \myref{lem:mod-p-qFs}.
  If $X$ is normal but has non-RDP singularities, the assertion follows from \myref{lem:qFs-elliptic-cone}.
  If $X$ is nonnormal, then the assertion follows from \myref{cor:gFs-irreducible} and \myref{cor:gFs-reducible}.
\end{proof}

\bibliographystyle{amsalpha}
\bibliography{mainref}

\end{document}